\documentclass[reqno]{amsart}

\usepackage{a4,amsmath,amssymb,amscd,verbatim,bbm,graphicx,enumerate}

\newtheorem{theorem}{Theorem}
\newtheorem{proposition}[theorem]{Proposition}
\newtheorem{corollary}[theorem]{Corollary}
\newtheorem{lemma}[theorem]{Lemma}

\theoremstyle{definition}

\newtheorem{definition}[theorem]{Definition}

\newtheorem{remark}[theorem]{Remark}

\numberwithin{equation}{section}
\numberwithin{theorem}{section}

\title{Statistical Limit Laws for Hyperbolic Groups}

\begin{document}
\bibliographystyle{plain}

\author{Stephen Cantrell}
\address{Mathematics Institute, University of Warwick,
Coventry CV4 7AL, U.K.}
\email{S.J.Cantrell@warwick.ac.uk}

\maketitle

\begin{abstract}
Using techniques from ergodic theory and symbolic dynamics, we derive statistical limit laws for real valued functions on hyperbolic groups. In particular, our results apply to convex cocompact group actions on $\text{CAT}(-1)$ spaces, and provide a precise statistical comparison between word length and displacement. After generalising our methods to the multidimensional setting, we prove that the abelianisation map satisfies a non-degenerate multidimensional central limit theorem. We also obtain local limit theorems for group homomorphisms and for the displacement function associated to certain actions.
\end{abstract}

\section{Introduction and Results}

Let $G$ be a non-elementary hyperbolic group with a fixed finite generating set $S$. Let $|g|$ denote the word length of $g \in G$ with respect to $S$ and write $W_n= \{g \in G: |g|=n\}$. There has been significant interest in understanding how the images of elements of $W_n$, under natural real valued maps, such as group homomorphism or quasimorphisms, are distributed in $\mathbb{R}$. For example, Horsham and Sharp proved that when $G$ is a free group, $S$ is the canonical free generating set for $G$ and $\varphi: G \to \mathbb{R}$ is a sufficiently regular quasimorphism (i.e. a group homomorphism up to bounded error), then the normalized images
$$\left\{ \frac{\varphi(g)}{\sqrt{n}} : g \in W_n\right\},$$
converge to a normal distribution as $n\to \infty$ \cite{horshamthesis}, \cite{HS}. Other similar statistical results have been proved when $G$ is a free group, \cite{GR}, \cite{45}, \cite{riv}.\\ \indent
In \cite{cf} Calegari and Fujiwara obtain a Gaussian limit law that holds for general non-elementary hyperbolic groups. They construct a sequence of measures $\nu_n$ on $G$, such that if $\varphi : G \to \mathbb{R}$ belongs to a class of functions, called bicombable functions, then there exists $A \in \mathbb{R}$ such that the distributions
$$\nu_n\left\{g \in G: \frac{\varphi(g) - An}{\sqrt{n}} \le x\right\},$$
converge as $n \to \infty$ to a normal distribution. Calegari extends this result in his survey \cite{calnotes}, showing that the above central limit theorem holds for a wider class of functions than bicombable functions.\\ \indent
The proof of these results rely on ideas and techniques from ergodic theory. In fact, these proofs follow a similar methodology that we now briefly outline. The general idea is to use the fact that hyperbolic groups are strongly Markov. That is, there exists a finite directed graph that in some sense encodes the key properties of $G$. Using this graph we can associate to the pair $G, \varphi$ (where $G$ is a hyperbolic group and $\varphi: G \to \mathbb{R}$ is in the required class) a dynamical system $(\Sigma, \sigma: \Sigma \to \Sigma)$ and a suitable function $f: \Sigma \to \mathbb{R}$. The function $f$ is chosen in such a way that the statistical behaviour of $\varphi$ on $G$ can be deduced from the statistical behaviour of $f$ on $(\Sigma,\sigma)$. Then, using techniques from ergodic theory, one can study the behaviour of $f$ on $\Sigma$ to deduce a central limit theorem for $\varphi$ on $G$. In the result of Calegari and Fujiwara, the measures $\nu_n$ are supported on $W_n$ and weight elements of $W_n$ by a quantity depending on the system $(\Sigma,\sigma)$. The system $(\Sigma,\sigma)$ associated to $G$ is not canonical and hence neither are the measures $\nu_n$. \\ \indent
The above discussion leads to the following natural questions.

\begin{enumerate}
\item Does the result of Horsham and Sharp generalise to the case that $G$ is an arbitrary non-elementary hyperbolic group?
\item In the result of Calegari and Fujiwara, can we replace the sequence  $\nu_n$ with a sequence of measures that does not depend on $(\Sigma,\sigma)$? In particular, can we replace $\nu_n$ with the sequence of uniform measures on $W_n$?
\end{enumerate}

In this paper we answer these questions in the affirmative. We also prove an averaging theorem and a large deviation theorem. We will now state these results, beginning with a discussion of the hypotheses we impose. Throughout the paper, $G$ will denote a non-elementary hyperbolic group.\\ \indent
We employ similar methods to those used in \cite{cf} and \cite{horshamthesis}. That is, given $G$ equipped with a finite generating set $S$, we construct a dynamical system $(\Sigma,\sigma)$ and embed $G$ into $\Sigma$ via a function $i: G \to \Sigma$. Specifically, $(\Sigma,\sigma)$ will be a subshift of finite type that encodes the key properties of $G$ and $S$. This will allow us to apply techniques from the branch of ergodic theory known as thermodynamic formalism. We are interested in the statistics of functions $\varphi: G \to \mathbb{R}$ that satisfy two conditions that appear as Condition $(1)$ and Condition $(2)$ below. Since these conditions are somewhat technical, we defer their precise statement until Section 3. Roughly speaking, these conditions allow us to translate questions about $\varphi$ on $G$ to questions about a suitable function $f: \Sigma \to \mathbb{R}$.  For now we note that there are many natural examples of functions satisfying these conditions, including group homomorphism, some quasimorphisms and as we will discuss shortly, the displacement function associated to certain group actions on $\text{CAT}(-1)$ spaces.\\
\indent Recall that, given $G$ and $S$, $W_n = \{g \in G: |g|=n\}$. We write $\# W_n$ to denote the cardinality of $W_n$. We now state our main results. 
 
\begin{theorem} [Averaging Theorem] \label{averaging}
Let $G$ be a non-elementary hyperbolic group equipped with a fixed generating set. Suppose that $\varphi: G \to \mathbb{R}$ satisfies Condition $(1)$ and Condition $(2)$. Then, there exists $\Lambda \in \mathbb{R}$ such that
$$ \frac{1}{\# W_n} \sum_{g \in W_n} \frac{\varphi(g)}{n} = \Lambda + O\left(\frac{1}{n}\right)$$
as $n\to\infty$. 
\end{theorem}
 
This result can be seen as an analogue of the law of large numbers, as it describes how $\varphi(g)/|g|$ averages over the sets $W_n$ as $n\to\infty$. This leads us to ask if we can describe more precisely how $\varphi$ averages over $W_n,$ as $n\to\infty$. If we assume that $\varphi(\cdot) - \Lambda |\cdot| : G \to \mathbb{R} $ is an unbounded function, where $\Lambda$ is the constant from Theorem \ref{averaging}, then we can deduce a central limit theorem for the normalised images 
$$\left\{\frac{\varphi(g)-n\Lambda}{\sqrt{n}}: g \in W_n\right\}.$$
Furthermore, using Theorem \ref{averaging} we quantify the rate of convergence associated to our central limit theorem. We show that the sequence of distributions that we consider converges uniformly to the Gaussian distribution at a $O\left(n^{-1/2}\right)$ rate. This is the so-called Berry-Esseen error term. 

\begin{theorem} [Central Limit Theorem] \label{CLT}
Let $G$ be a non-elementary hyperbolic group equipped with a finite generating set. Suppose that $\varphi: G \to \mathbb{R}$ satisfies Condition $(1)$ and Condition $(2)$. Let $\Lambda$ be the constant from Theorem \ref{averaging} and suppose that $\varphi(\cdot) - \Lambda |\cdot|:G \to \mathbb{R}$ is unbounded. Then, there exists $\sigma^2 > 0$ such that
$$\frac{1}{\# W_n} \# \left\{ g \in W_n:  \ \frac{\varphi(g) - n\Lambda}{\sqrt{n}} \le x\right\} = \frac{1}{\sqrt{2\pi}  \sigma} \int_{-\infty}^x e^{-t^2/2\sigma^2} \ dt + O\left(\frac{1}{\sqrt{n}}\right),$$
where the implied constant is independent of $x \in \mathbb{R}$.
\end{theorem}
We also prove the following large deviations result.
\begin{theorem} [Large Deviation Theorem] \label{ldt}
Let $G$ be a non-elementary hyperbolic group. Suppose that $\varphi: G \to \mathbb{R}$ satisfies Condition $(1)$ and Condition $(2)$. Then, for any $\epsilon >0$,
$$\limsup_{n \to \infty} \ \frac{1}{n} \log \left( \frac{1}{\# W_n} \#\left\{g \in W_n : \left|\frac{\varphi(g)}{n} - \Lambda \right| > \epsilon \right\} \right)< 0,$$
where $\Lambda$ is as in Theorem \ref{averaging}.
\end{theorem}
We will show that Theorem \ref{CLT} provides a positive answer to the two question posed earlier in this introduction. Apart from answering these two questions, our motivation behind this work is to understand the statistics of the displacement function associated to group actions on $\text{CAT}(-1)$ spaces. We are interested in answering the following question.\\ \indent 
Let $(X,d)$ be a complete $\text{CAT}(-1)$ geodesic metric space and fix an origin $o$ for $X$. Suppose that a hyperbolic group $G$ equipped with a finite generating set acts on $X$ properly discontinuously, convex cocompactly by isometries. The \u{S}varc-Milnor Lemma implies that there exists constants $C_0,C_1 > 0$ such that
$$C_0 |g|\le d(o,go) \le C_1 |g|$$
for all $g \in G$. We call the function $g \mapsto d(o,go)$ the displacement. The above inequality shows that word length and displacement are comparable quantities. This leads us to ask whether we can form a more refined comparison, on average, between them? \\ \indent
We show that the displacement function satisfies Condition $(1)$ and Condition $(2)$. Theorems \ref{averaging}, \ref{CLT} and $1.3$ then apply and we obtain the following comparison results.

\begin{theorem}
Suppose a non-elementary hyperbolic group $G$ acts convex cocompactly, properly discontinuously by isometries on a complete, geodesic, $\text{CAT}(-1)$ metric space $(X,d)$. Fix an origin $o \in X$ and a finite generating set for $G$. Then there exists $\Lambda >0 $ such that
$$ \frac{1}{\# W_n} \sum_{g \in W_n} \frac{d(o,go)}{n} = \Lambda + O\left(\frac{1}{n}\right).$$
Also, for any fixed $\epsilon >0$,
$$\limsup_{n \to \infty} \ \frac{1}{n} \log \left( \frac{1}{\# W_n} \#\left\{g \in W_n : \left|\frac{d(o,go)}{n} - \Lambda \right| > \epsilon \right\} \right)< 0.$$
Furthermore, if $d(o,\cdot \hspace{1pt} o)-\Lambda |\cdot|:G \to \mathbb{R}$ is unbounded then there exists $\sigma^2 > 0$ such that
$$\frac{1}{\# W_n} \# \left\{ g \in W_n:  \ \frac{d(o,go) - n\Lambda}{\sqrt{n}} \le x\right\} = \frac{1}{\sqrt{2\pi} \sigma} \int_{-\infty}^x e^{-t^2/2\sigma^2} \ dt + O\left(\frac{1}{\sqrt{n}}\right),$$
where the implied constant is independent of $x \in \mathbb{R}$.
\end{theorem}

\begin{remark}
We note that similar results have been obtained by Gekhtman, Taylor and Tiozzo in \cite{lox} and \cite{gtt}.\\

\noindent (i) In [11], Gekhtman, Taylor and Tiozzo showed that 
\begin{equation}
\frac{1}{\# W_n} \#\left\{g \in W_n : \left|\frac{d(o,go)}{n} - \Lambda \right| > \epsilon \right\} \to 0
\end{equation}
with no estimate on the rate of convergence, for non-elementary actions (see Section $5$ of \cite{lox} for a definition) of $G$ on hyperbolic metric spaces. These actions are more general than convex cocompact actions. However, we have recently learned from these authors that the random walk results they used have been improved by Sunderland \cite{sun} and that this improvement, combined with the work in \cite{lox}, gives exponential convergence in $(1.1)$ at the level of generality considered in \cite{lox}.\\ 
\noindent (ii) In \cite{gtt}, Gekhtman, Taylor and Tiozzo obtained a central limit theorem as above (but without an error term) in the special case where G is a free group or surface group.
\end{remark}

After proving the above results, we generalise our method to the multidimensional setting with the aim of studying the statistics of the abelianisation homomorphism $\varphi : G \to G/[G,G]$. The abelianisation $G/[G,G]$ takes the form $\mathbb{Z}^k \oplus \text{Torsion }$ for some $k \ge 0$ and we are interested in how the image of $G$ distributes in the non-torsion factor, $\mathbb{Z}^k$. We will assume that $k\ge1$ and that we have fixed an isomorphism taking the non-torsion part of $G/[G,G]$ to $\mathbb{Z}^k$. We will refer to the induced homomorphism $\varphi: G \to \mathbb{Z}^k$ as the abelianisation homomorphism. Note that the components of this map are integer valued homomorphisms and so satisfy Condition $(1)$ and Condition $(2)$. This will allow us to apply the multidimensional analogues of the methods used to prove Theorems \ref{averaging}, \ref{CLT} and $1.3$. We prove that the abelianisation homomorphism satisfies a non-degenerate multidimensional central limit theorem.

\begin{theorem}
Let $G$ be a non-elementary hyperbolic group equipped with a finite generating set $S$. Suppose that $G$ has abelianisation $\mathbb{Z}^k \oplus \text{Torsion }$ for some $k \ge 1$ and that $\varphi: G \to \mathbb{Z}^k$ is the abelianisation homomorphism constructed in the way described above. Then there exists a symmetric, positive definite matrix $\Sigma \in M_k(\mathbb{R})$ such that for $A \subset \mathbb{R}^k$,
$$\frac{1}{\#W_n} \# \left\{ g \in W_n : \frac{\varphi(g)}{\sqrt{n}} \in A \right\} \rightarrow \frac{1}{(2\pi \ \det(\Sigma))^{k/2}} \int_{A} e^{-\langle x , \Sigma x\rangle /2} \ dx $$
as $n\to\infty$.
\end{theorem}
\noindent This result generalises the work of Rivin, who, in \cite{riv}, proves the above theorem for free groups.\\ \indent

In the last section, we consider more subtle distributional results than those in the above theorems. That is, we prove local central limit theorems. To obtain these results we need to understand the arithmetic properties of the images of the functions $\varphi: G \to \mathbb{R}$ that we consider. To gain this understanding, we need to assume that  $\varphi$ satisfies some additional properties. Our methods and therefore results do not apply to all maps satisfying Condition $(1)$ and Condition $(2)$. We obtain the following local limit theorem for group homomorphisms to $\mathbb{R}$.
\begin{theorem}\label{ndllt}
Suppose $G$ is a non-elementary hyperbolic group equipped with a finite generating set. Let $\varphi:G\to\mathbb{R}$ be a group homomorphism that has a dense image in $\mathbb{R}$. Then, Theorem \ref{CLT} holds and we obtain $\sigma >0$ such that any $a,b \in \mathbb {R}$ with $a<b$,
$$\frac{1}{\#W_n} \# \{ g \in W_n: \ \varphi(g)  \in [a,b]\} \sim \frac{b-a}{\sqrt{2\pi}\sigma \sqrt{n}}$$
as $n \to \infty$.
\end{theorem}
In \cite{sharplimit} Sharp studies local limit theorems for homomorphisms $\varphi:G \to \mathbb{Z}$ (where $G$ is a free group). In this work we are interested in the complementary case in which the image of $\varphi : G \to \mathbb{R}$ is dense in $\mathbb{R}$. We show that, in a natural sense, almost all group homomorphisms satisfy the hypotheses of Theorem \ref{ndllt}. After proving this result we obtain a further local limit theorem for the displacement function associated to convex cocompact actions on pinched Hadamard surfaces. We defer the statement of this result until the last section.\\ \indent

To conclude the introduction, we briefly outline the contents of this paper. In the second section we cover preliminary material from thermodynamic formalism and geometric group theory. In the third section, we study the regularity conditions required for functions to satisfy our theorems. In Section 4 we introduce and study the spectral theory of certain transfer operators. This is where we appeal to the work of Calegari and Fujiwara \cite{cf} to deduce key properties of these operators. We then, in Section $5$, study positive variance condition required for our central limit theorem. The subsequent sections are dedicated to proving our results.\\ \indent
The work presented in this paper will form part of the author's PhD thesis at the University of Warwick. The author would like to thank Richard Sharp for useful discussions, comments and suggestions.\\ 

\noindent \textit{Notation:} Throughout the paper we use the following notation to describe the asymptotic behaviour of sequences. For $f,g : \mathbb{Z}_{\ge 0} \rightarrow \mathbb{R}$, we write $f \sim g$ if $f(n)/g(n) \to 1$ as $n \to \infty$. If $f(n)/g(n) \to 0$ as $n \to \infty$ then we write $f = o(g)$ and if there exists $K>0$ such that eventually $|f(n)| \le K|g(n)|$, then we write $f = O(g)$. If $f = O(g)$ and $g=O(f)$ we write $f = \Theta(g)$.

\section{Preliminaries}

\subsection{Thermodynamic Formalism and Subshifts of Finite Type}
We recall some basic material from the branch or ergodic theory known as thermodynamic formalism. For a comprehensive account see \cite{pp}. Let $A$ be a $k \times k$ matrix consisting of zeros and ones. We use the notation $A_{i,j}$ to denote the $(i,j)$th entry of $A$.
\begin{definition}
We say that a $k \times k$ zero-one matrix $A$ is irreducible if given $(i,j)$ ($i,j \in \{1,2,...,k\}$), there exists $n\in \mathbb{N}$ such that $(A^n)_{i,j} >0$. We say that $A$ is aperiodic if there exists $n \in \mathbb{N}$ such that $(A^n)_{i,j} > 0$ for all $i,j$. 
\end{definition}
The shift space, or subshift of finite type, associated to $A$, is the space
$$\Sigma_A = \{(x_n)_{n=0}^{\infty} : x_n \in \{1,2,...,k\}, A_{x_n,x_{n+1}}=1, n \in \mathbb{Z}_{\ge 0}\}.$$
Given $x$ in $\Sigma_A$, $x_n$ denotes the $n$th coordinate of $x$. The shift map $\sigma: \Sigma_A \rightarrow \Sigma_A$ sends $x$ to $y= \sigma(x)$ where $y_n=x_{n+1}$ for all $n \in \mathbb{Z}_{\ge 0}$. For each $0<\theta <1$, we define a metric $d_\theta$ on $\Sigma_A$. Take $x,y \in \Sigma_A$.  If $x_0 = y_0$ we set $d_\theta(x,y)= \theta^N,$ where $N$ is the largest positive integer such that $x_i = y_i$ for all $0 \le i<N$. If $x_0 \neq y_0$ we set $d(x,y)=1$. We then define
$$F_\theta(\Sigma_A) = \{f:\Sigma_A \rightarrow \mathbb{C}:\text{$f$ is Lipschitz with respect to $d_\theta$}\}.$$
For the rest of this section we denote $F_\theta(\Sigma_A)$ by $F_\theta$ to simplify notation. We equip $F_\theta$ with the norm $\|f\|_\theta = |f|_\theta + |f|_\infty,$ where $|f|_\infty$ is the sup-norm and $|f|_\theta$ denotes the least Lipschitz constant for $f$. The space $(F_\theta, \|\hspace{2mm} \|_\theta)$ is a Banach space. We say that $f,g \in F_\theta$ are cohomologous (denoted by $f \sim g$) if there exists a continuous function $h:\Sigma_A \rightarrow \mathbb{C}$ such that $f=g+h\circ \sigma - h.$ We have the following characterisation for a function $f \in F_\theta$ being cohomologous to a constant. The set $\{f^n(x) - Cn: x\in \Sigma_A\}$ is bounded if and only if $f$ is cohomologous to $C$. Here $f^n(x) := f(x) + f(\sigma(x))+...+f(\sigma^{n-1}(x))$.\\ \indent
Suppose that $A$ is aperiodic. Given $f \in F_\theta$, we define the pressure of $f$ by
$$P(f) = \sup_{m} \left\{ h_m(\sigma) + \int f dm \right\},$$
where $h_m(\sigma)$ denotes the entropy of $\sigma$ with respect to $m$ and the supremum is taken over all $\sigma$-invariant probability measures. This supremum is uniquely attained by a measure called the equilibrium state of $f$. The equilibrium state for the zero constant function is the measure of maximal entropy and the topological entropy of $(\Sigma,\sigma)$ is given by $P(0)$. For real valued $f \in F_\theta$ the map $s \mapsto P(sf)$ is real analytic and extends to a complex analytic function in a neighbourhood of the real line.\\ \indent
The following operators, known as transfer operators, play a key role in the analysis used in this paper.  
\begin{definition}
Take $f \in F_\theta$, we define the transfer operator $L_f: F_\theta \rightarrow F_\theta$ by
$$L_fw(x)=\sum_{\sigma y =x} e^{f(y)} w(y).$$
\end{definition}
The transfer operator has a variety of useful properties. We are interested in the spectral properties of these operators. In the case that $A$ is aperiodic and $f \in F_\theta$ is real valued, the spectrum of $L_f$ contains a real simple maximal eigenvalue $\lambda >0$. The rest of the spectrum is contained in the disk $\{z \in \mathbb{C}: |z|<\lambda - \delta\}$ for some $\delta>0$. Using the spectral properties of $L_f$ and perturbation theory, one can show that
$$\frac{dP(sf)}{ds}  \Big|_{s=0} = \int f \ d\mu$$
and
$$\sigma^2 := \frac{d^2P(sf)}{ds^2}  \Big|_{s=0} = \lim_{n \to \infty} \frac{1}{n} \int \left(f^n - \int f d\mu\right)^2 d\mu ,$$
where $\mu$ is the measure of maximal entropy. The quantity $\sigma^2$ is strictly positive if and only if $f$ is not cohomologous to a constant. Using these identities one obtains the following expression for pressure (see \cite{cp}), which is valid for complex $s$ in a neighbourhood $U$ of $0$:
\begin{equation} \label{eq:1}
P(sf) = P(0) + s \int f d\mu +  s^2\sigma^2/2 + s^3 \psi(s),
\end{equation}
 where $\psi$ is analytic in $U$.\\ \indent
Now suppose that $A$ is irreducible but not aperiodic. There exists a natural number $p>1$ known as the period of $A$ such that $\Sigma_A$ has $p$-cyclic decomposition
$$\Sigma_A = \bigsqcup_{k=0}^{p-1} \Sigma_{A_k}.$$
The shift map sends $\Sigma_{A_j}$ to $\Sigma_{A_{j+1}}$ where $j, j+1$ are taken modulo $p$. Furthermore, for each $j$, $\sigma^p : \Sigma_{A_j} \to \Sigma_{A_j}$ is mixing. The transfer operator $L_0: F_\theta \to F_\theta$ (i.e. where $0$ denotes the zero valued constant function) has spectrum containing $p$ simple maximal eigenvalues at $e^{2\pi i k/p} e^h$ for $k = 0,...,p-1$. The rest of the spectrum is contained in the disk $\{z:|z|<e^{h}-\delta\}$ for some $\delta>0$. The constant $h$ is the topological entropy of $(\Sigma,\sigma)$ and is obtained, as in the case when $A$ is aperiodic, from the variational expression
$$h = \sup_{m}\{h_m(\sigma)\},$$
where the above supremum is taken over all $\sigma$-invariant probability measures. This supremum is attained uniquely by the measure of maximal entropy.

\subsection{Hyperbolic Groups and the Strongly Markov Property}
In this section we recall classical properties of hyperbolic groups. The concept of hyperbolicity was introduced by Gromov in his fundamental paper \cite{grom}. For a good account of the theory concerning hyperbolic groups, see \cite{gh}.
\begin{definition}
Let $(X,d)$ be a metric space. We say that $X$ is hyperbolic if there exists a constant $\delta \ge 0$ such that given any geodesic triangle $xyz$ in $X$, the side $xy$ is contained in the union of the $\delta$-neighbourhoods of the other two sides, $yz$ and $zx$.
A finitely generated group $G$ is hyperbolic (in the sense of Gromov) if for any finite generating set $S$ for $G$, the Cayley graph of $G$ with respect to $S$ is hyperbolic when equipped with the path metric. 
\end{definition}
A hyperbolic group is non-elementary if it is not virtually cyclic, i.e. it does not contain a finite index cyclic subgroup. In this paper we are only interested in non-elementary hyperbolic groups. All groups labeled $G$ are assumed to be non-elementary hyperbolic groups. Given an element $g \in G$, we use $|g|$ to denote the word length of $g$: the length of the shortest word(s) representing $g$ with letters in $S\cup S^{-1}$. Let $W_n$ denote the set consisting of group elements of word length $n$. We define the left and right word metrics on $G$ as follows.
\begin{definition}
The left and right word metrics on $G$ are
$$d_L(g,h) = |g^{-1}h| \hspace{0.5 cm} \text{and} \hspace{0.5cm} d_R(g,h)= |gh^{-1}|$$
respectively.
\end{definition}
Throughout this paper we will require some techniques from Patterson-Sullivan theory. We recall some basic facts about the boundaries of hyperbolic groups and the Patterson-Sullivan measure. \\
\indent Let $C(G)$ denote the Cayley graph of $G$ with respect to $S$. An infinite geodesic ray $\gamma$ is an infinite path in $C(G)$ such that any finite sub-path of $\gamma$ is a geodesic in $C(G)$. Given such a geodesic ray $\gamma$, let $\gamma_n$ denote the element in $G$ corresponding to the end point of $\gamma$ after $n$ steps. Two geodesic rays $\gamma, \gamma'$ are said to be equivalent if $d_L(\gamma_n,\gamma'_n)$ is bounded uniformly for $n \in \mathbb{Z}_{\ge 0}$. The Gromov boundary, $\partial G$ of $G$, is the set of equivalence classes of infinite geodesic rays in $C(G)$. The boundary $\partial G$ supports a natural (metrizable) topology. With this topology, $G \cup \partial G$ becomes the compactification of $G$ (with the topology given by the word metric). Given a geodesic ray $\gamma$ let $[\gamma] \in \partial G$ denote the equivalence class containing $\gamma$. The action of $G$ extends to $G \cup \partial G$ by sending $[\gamma] \in \partial G$ to $[g\gamma] \in \partial G$. \\
\indent The Patterson-Sullivan measure $\nu$ is a measure on $G \cup\partial G$ that is supported on $\partial G$. It is obtained as a weak star limit, as $n \to \infty$, of the following sequence of measures
$$\frac{\sum_{|g|\le n} \lambda^{-|g|}\delta_g}{\sum_{|g|\le n} \lambda^{-|g|}}.$$
Here $\lambda =\limsup_{n\to \infty} (\#W_n)^{1/n}$ is the exponential growth rate of $\#W_n$ and $\delta_g$ denotes the Dirac measure based at $g \in G$.
The measure $\nu$ enjoys many useful properties and in particular is ergodic with respect to the action of $G$ on $\partial G$. For a comprehensive account of the above material, see \cite{coor} and \cite{patcat}. \\

 \indent Hyperbolic groups have interesting combinatorial properties. One of the reasons for this is their strongly Markov structure: a hyperbolic group can be represented by a finite directed graph with useful properties.

\begin{definition}
A group $G$ is strongly Markov if given any generating set $S$ for $G$, there exists a finite directed graph $\mathcal{G}$ with vertex set $V$ and directed edge set $E \subset V \times V$ that exhibits the following properties:
\begin{enumerate}
\item $V$ contains a vertex $\ast$ such that $(x,\ast)$ does not belong to $E$ for any $x \in V$,
\item there exists a labeling $\rho :E \to S$ such that the map sending a path (starting at $\ast$) with concurrent edges $(\ast,x_0),(x_0,x_1),\ldots , (x_{n-1},x_n)$ to the group element $\rho(\ast,x_0) \rho(x_0,x_1) \ldots \rho(x_{n-1},x_n),$ is a bijection,
\item the above bijection preserves word length; if $|g|=n$, then the finite path corresponding to $g$ has length $n$.
\end{enumerate}
\end{definition}
We augment the above directed graph by adding an extra vertex, $0$. We add directed edges from each vertex in $V \backslash \{\ast\}$ to $0$ and also from $0$ to itself. We extend the labeling $\lambda$ to these new edges by $\rho(x,0) = e$ (the identity element in $G$) for all $x \in V \cup \{0\} \backslash \{\ast\}$.\\ \indent
Cannon proved that cocompact Kleinian groups are strongly Markov. Ghys and de la Harpe showed that Cannon's approach worked for all hyperbolic groups. The augmentation method described above was first used by Lalley \cite{lal} to facilitate the use of thermodynamic formalism.
\begin{proposition} [Cannon \cite{can}, Ghys and de la Harpe \cite{gh}]
Any hyperbolic group is strongly Markov.
\end{proposition}
Throughout the rest of this paper, given a hyperbolic group $G$ with generating set $S$, we use $\mathcal{G}$ to denote a directed graph associated to $G$ via the strongly Markov property. We will always assume that such $\mathcal{G}$ has been augmented, to include the $\ast$ and $0$ vertices, in the way described above. We note that $G$ can admit infinitely many different graphs satisfying the properties in Definition $2.5$.\\ \indent
This strongly Markov structure makes hyperbolic groups susceptible to analysis through the use of thermodynamic formalism and subshifts of finite type. Let $G$ be a hyperbolic group with associated directed graph $\mathcal{G}$. Labeling the vertices of $\mathcal{G}$, $1,...,k$, we can describe $\mathcal{G}$ by a $k\times k$ zero-one matrix $A$. We set the $(i,j)$th entry of $A$ to be $1$ if and only if there exists an edge from vertex $i$ to vertex $j$. We call $A$ the transition matrix associated to $\mathcal{G}$. We can then embed $G$ into the shift space $\Sigma_A$ via the function $i: G  \to \Sigma_A$ defined by
$$i(g) = (\ast,x_0,x_1,\ldots,x_{n-1},0,0,\ldots),$$
where  $(\ast,x_0),(x_0,x_1),\ldots , (x_{n-2},x_{n-1})$ is the unique shortest path in $\mathcal{G}$ corresponding to $g$ and $|g| = n$. We use the notation $\dot{0}$ to denote the sequence in $\Sigma_A$ consisting of only zeros.\\ \indent
Property $(3)$ from Definition $2.5$ implies that the cardinality of $W_n$, denoted by $\# W_n$, is given by the number of length $n$ paths in $\mathcal{G}$ starting at $\ast$. Coornaert proved that the growth of $\# W_n$ is purely exponential \cite{coor}.

\begin{proposition} [Coornaert \cite{coor}]
There exists $C_1,C_2 >0$, $\lambda>1$ such that for all $n \ge 1$,
$$  C_1 \lambda^n \le \# W_n \le C_2 \lambda^n.$$
\end{proposition}
Let $B$ denote the matrix $A$ with the columns and rows corresponding to the $\ast$ and $0$ vertices removed.

\begin{definition}
Let $\mathcal{G}$, $A$ and $B$ be as above. We say that $\mathcal{G}$ is aperiodic (or irreducible) if $B$ is aperiodic (or irreducible).
\end{definition}

In general, it is possible that $\mathcal{G}$ is not irreducible. However, in certain cases, for example for surface groups with presentation $\langle a_1,\ldots,a_g,b_1,\ldots,b_g | \prod_{j=1}^g [a_j,b_j] \rangle$ and free groups equipped with their canonical free generating set, $\mathcal{G}$ can be chosen to be aperiodic \cite{cs}. When $\mathcal{G}$ is aperiodic, results from thermodynamic formalism apply more readily. One of the main difficulties in this paper is overcoming the extra difficulties that arise in the case that $\mathcal{G}$ is neither aperiodic nor irreducible\\
\indent Following  \cite{oc}, we can relabel the columns/rows of $B$ to assume that $B$ has the form 

$$B = \begin{pmatrix} 
B_{1,1} & 0 & \dots & 0  \\
B_{2,1} & B_{2,2} & \dots & 0\\
\vdots & \vdots & \ddots & \vdots\\
B_{m,1} & B_{m,2} & \dots & B_{m,m}
\end{pmatrix} ,$$

\noindent where the matrices $B_{i,i}$ are irreducible. The matrices $B_{i,i}$ are known as the irreducible components of $B$ or $\mathcal{G}$. By property $(3)$ in Definition $2.5$ and Proposition $2.7$, the spectral radius of each $B_{i,i}$ is bounded above by $\lambda$, where $\#W_n=\Theta(\lambda^n)$. Furthermore, there must be at least one component that has $\lambda$ as an eigenvalue (otherwise there would be $0<\delta<\lambda$ for which $\# W_n = O((\lambda-\delta)^n)$).

\begin{definition}
We call an irreducible component maximal if its corresponding matrix has spectral radius $\lambda$. 
\end{definition}

An important property of $\mathcal{G}$ is the following.

\begin{lemma} \cite[Lemma 4.10]{cf}
Let $\mathcal{G}$ be a directed graph associated to $G$. The maximal components of $\mathcal{G}$ are disjoint. That is, there does not exist a path in $\mathcal{G}$ from one maximal component to another.
\end{lemma}
\begin{proof}
Let $B_1$ and $B_2$ be maximal components and suppose there is a path $P$ of length $l$ from $B_1$ to $B_2$. Then for $n >l$, the number of length $n$ paths in $\mathcal{G}$ would be at least 
\begin{equation}
 \sum_{r + s = n-l}B_1^r B_2^s,
 \end{equation}
 where $B_1^k$ denotes the number of length $k$ paths contained in $B_1$ ending at the start vertex of $P$ and $B_2^k$ denotes the number of length $k$ paths in $B_2$ starting at the end vertex of $P$. Quantity $(2.2)$ grows like  $n\lambda^n$ which implies that $\#W_n$ grows at least like $n \lambda^n$. This contradicts Proposition $2.7$.
\end{proof}

\section{Regularity for Functions}

In this section we discuss the regularity conditions required for functions to satisfy our theorems. Fix a generating set $S$ for $G$.  Recall that we are interested in functions $\varphi:G \to \mathbb{R}$ that satisfy Condition $(1)$ and Condition $(2)$. These conditions are defined as follows. \medskip \\
 \noindent
\textit{Condition $(1)$} There exists a directed graph $\mathcal{G}$ associated to $G,S$ via the strongly Markov property with transition matrix $A$ and a function $f\in F_\theta(\Sigma_A)$ (for some $0<\theta<1$) such that $\varphi(g) = f(x) + f(\sigma(x)) +...+f(\sigma^{|g|-1}(x))$ for $g \in G \text{ and } x = i(g) \in \Sigma_A.$\medskip \\
 \noindent
\textit{Condition $(2)$} $\varphi$ is Lipschitz in the left and right word metrics on $G$. \medskip

We begin this section by discussing examples of functions that satisfy Condition $(1)$. The first class we consider, is of functions that satisfy the following H\"older condition.
\begin{definition}
We say that a map $\varphi:G \to \mathbb{R}$ is H\" older (for $G$, $S$) if for any fixed $a \in G$ there exists $C  >0$ and $0<\theta <1$ such that
$$|\Delta_a\varphi(g)-\Delta_a\varphi(h)| \le C \theta^{ (g,h)},$$
for any $g,h \in G$. Here, for $a,g \in G$, $\Delta_a\varphi (g)= \varphi(ag)-\varphi(g)$ and $(g,h)$ denotes the Gromov product of $g$ and $h$,
$$(g,h)=\frac{1}{2}\left(|g|+|h| - |g^{-1}h|\right).$$
\end{definition}

Pollicott and Sharp proved that any function satisfying the above H\"older condition for $G$, $S$, satisfies Condition $(1)$ (see Lemma $1$ of \cite{oc}). In fact, they showed that for such functions, one can find an appropriate H\"older function $f: \Sigma_A \to \mathbb{R}$ given any graph $\mathcal{G}$ associated to $G,S$. Inspired by the work of Calegari and Fujiwara, we introduce the following class of functions.
\begin{definition}
Suppose $S$ is symmetric. Given an element $g \in G$, there is a unique path of length $|g|$ in $\mathcal{G}$, starting at $\ast$, that is mapped to $g$ under the bijection defined in part $(3)$ of Definition $2.5$. Let $g_i$ belong to the edge set of $\mathcal{G}$ and let it denote the $i$th edge in the path corresponding to $g$. A map $\varphi: G \to \mathbb{R}$ is called edge combable (with respect to $\mathcal{G}$) if there exists a function $d\varphi$ from the edge set of $\mathcal{G}$ to $\mathbb{R}$ such that, for each $g \in G$,
$$\varphi(g) = \sum_{i=1}^{|g|} d\varphi(g_i).$$
We refer to $d\varphi$ as the (discrete) derivative of $\varphi$.
\end{definition}

\begin{remark}
In \cite{cf} Calegari and Fujiwara define the class of combable functions. These functions are similar to edge combable functions except that the derivative $d\varphi$  is a function from the vertex set of $\mathcal{G}$ to $\mathbb{Z}$. The equation relating $\varphi$ and $d\varphi$ is the same except the sum is taken over the vertices in the path corresponding to $g$. Given a combable function $\varphi$, one can consider $\varphi$ as an edge combable function. To see this, take the derivative $d\varphi$ of $\varphi$ (which is a function defined on the vertex set of $\mathcal{G}$) and define $d\varphi'$ on the edge set of $\mathcal{G}$ to send a directed edge to the value of $d\varphi$ evaluated at the end point of this edge. It is easy to see that $\varphi$ can be considered an edge combable function with derivative $d\varphi'$. Therefore the set of edge combable functions contains the set of combable functions.
\end{remark}

\begin{remark}
Suppose that $\varphi$ is edge combable with respect to $\mathcal{G}$ and that $d\varphi$ is integer valued. Then, we can find a different directed graph $\mathcal{G}'$ that satisfies the properties in Definition $2.5$ and for which $\varphi$ is combable. To see this, consider the following recoding of $\mathcal{G}$ to $\mathcal{G}'$. Define the vertex set for $\mathcal{G}'$ to be the edge set of $\mathcal{G}$ and say that two vertices $u$ and $v$ in $\mathcal{G}'$ are connected by a directed edge from $u$ to $v$ if the edges $e$, $r$ in $\mathcal{G}$ corresponding to $u,v$ are concurrent in $\mathcal{G}$. This process may introduce multiple $\ast$ vertices for $\mathcal{G}'$, however, we can simply identify these vertices to overcome this problem. 
\end{remark}

The above discussions imply that the class of edge combable functions includes combable functions and real valued homomorphisms.

\begin{lemma}
Edge combable functions satisfy Condition $(1)$.
\end{lemma}
\begin{proof}
 Let $\varphi$ be edge combable with derivative $d\varphi$. For $x = (x_n)_{n=0}^{\infty}\in \Sigma_A$, define
\[
  f(x) = \left\{
     \begin{array}{@{}l@{\thinspace}l}
        d\varphi((x_0,x_1)) & \ \  x_1 \neq 0 \\
        0 & \ \ x_1 =0,
     \end{array}
   \right.
\]
where $\rho$ denotes the labeling map defined in Definition $2.5$.
Since $f$ is constant on cylinders of length $2$, $f \in F_\theta(\Sigma_A)$ for any $0<\theta<1$. To see that Condition $(1)$ is satisfied, note that
\begin{align*}
f^{|g|}(i(g)) &= \sum_{k=0}^{n-1} f(\sigma^k(\ast,y_0,...,y_{n-1},\dot{0}))\\
&= \sum_{i=1}^n d\varphi(g_i)\\
&= \varphi(g).
\end{align*} 
\end{proof}
We have now seen examples of functions that satisfy Condition $(1)$. A large class of functions that satisfy Condition $(2)$ are quasimorphisms.

\begin{definition}
A function $\varphi : G \to \mathbb{R}$ is a quasimorphism if there exists a constant $A >0$ such that
$$|\varphi(gh) - \varphi(g) -\varphi(h)|\le A$$ for all $g,h \in G$.
\end{definition}
\noindent It is a simple exercise to show that quasimorphisms satisfy Condition $(2)$. We note that functions satisfying Definition $3.1$ are always Lipschitz in the left word metric. This can be seen by setting $h$ to be $e$, the identity  of $G$, in Definition $3.1$. When $h=e$, $\Delta_a\varphi(h)=\varphi(a)-\varphi(e)$ for any $a \in S$ and hence $|\varphi(ag)-\varphi(g)| \le C + |\varphi(a)| + |\varphi(e)|$ for all $g\in G$. It is easy to see that this implies $\varphi$ to be Lipschitz in the left word metric.\\ \indent

We now consider the class of functions satisfying both Condition $(1)$ and Condition $(2)$. In \cite{cf} Calegari and Fujiwara consider combable functions that are Lipschitz in the left and right words metrics on $G$. Furthermore, they prove that the class consisting of these functions is independent of the choice of symmetric $S$ and $\mathcal{G}$ associated to $G$. Hence our results apply to all functions considered by Calegari and Fujiwara in \cite{cf}. In particular our results apply to Brooks counting quasimorphisms \cite{brooks}, \cite{ef}.  In \cite{horshamthesis}, \cite{HS} Horsham and Sharp consider H\"older quasimorphism, which as discussed above, satisfy Condition $(1)$ and Condition $(2)$. Hence our results also apply to these functions. There are many other examples of functions satisfying Conditions $(1)$ and $(2)$. See, for example, \cite{bg}, \cite{ef} and \cite{gh}. As discussed in the introduction, the following examples are of particular interest to us.\\ \indent
Let $(X,d)$ be a complete $\text{CAT}(-1)$ geodesic metric space. A group $G$ is said to act convex cocompactly on $X$ if the quotient of the intersection of X and the convex hull (in $X$) of the limit set of $G$, is compact. Suppose $G$ acts properly discontinuously, convex cocompactly by isometries on $X$. Fix a finite generating set for $G$ and an origin $o$ (in the convex hull of the limit set of $G$) for $X$. 

\begin{lemma}
In the setting described above, the displacement function $g \mapsto d(o,go)$ satisfies Condition $(1)$ and Condition $(2)$.
\end{lemma} 
\begin{proof}
The fact that the displacement satisfies Condition $(1)$ is due to Pollicott and Sharp. This was proved in \cite{18} (see Proposition $3$) when $G$ acts on a negatively curved manifold $X$. However, the only property of $X$ required for the proof is the $\text{CAT}(-1)$ property and hence the proof applies to our case also. Showing that Condition $(2)$ is satisfied is a simple exercise.
\end{proof}
We end this section by observing that we can relax Condition (1) to the assumption that the equality $\varphi(g) = f(x) + f(\sigma(x)) + \ldots + f(\sigma^{|g| -1}(x))$ for $x=i(g)$ holds up to an error that is uniformly bounded in $G$. This is because, if our statistical results (other than the local limit theorems) hold for $\varphi:G \to \mathbb{R}$ then they hold for any function $\phi : G \to \mathbb{R}$ that is in uniformly bounded away from $\varphi$, i.e. $\sup_{g \in G} |\varphi(g) - \phi(g)| < \infty$. This means that our results can be applied to random matrix products as in Theorem $3.1$ of \cite{45}.

\section{Transfer Operators and Spectral Theory}
\subsection{Describing the spectrum}
Let $G,S$ have associated directed graph $\mathcal{G}$ described by transition matrix $A$. To deduce our main results, we analyse the following weighted sum
$$ \sum_{g \in W_n} e^{s\varphi(g)},$$
for $s\in \mathbb{C}$ with $|s|$ small, as $n\to\infty$. We want to express this sum in terms of transfer operators. To form a useful expression, we exploit the structure of $\mathcal{G}$ and in particular, use the fact that maximal components are disjoint. We therefore consider transfer operators of a specific form. The aim of this section is to define and study these operators. 

\begin{definition}
For $f \in F_\theta(\Sigma_A)$ define the transfer operator $L_{A,f}:F_\theta(\Sigma_A) \to F_\theta(\Sigma_A)$ by
$$L_{A,f}g(x) = \sum_{\substack{\sigma(y)=x \\ y \in \Sigma_A \backslash \{\dot{0}\}}}e^{f(y)}g(y).$$
\end{definition}
Note that these transfer operators vary slightly from those defined in Definition $2.2$, as we are excluding $\dot{0}$ as a possible preimage in the sum defining the operators. Pollicott and Sharp studied the spectral properties of these operators in \cite{oc}.\\ \indent
Let $B_{i}$ for $i=1,...,m$ denote the maximal components of $A$. 
\begin{definition}
For each $i=1,...,m$, define a matrix $C_i$ by,
\[
  C_i(u,v) = \left\{
     \begin{array}{@{}l@{\thinspace}l}
          0& \ \  \text{if $u$ or $v$ belong to a maximal component that is not $B_i$,}\\
          A(u,v) & \ \ \text{otherwise}.
     \end{array}
   \right.
\]

\end{definition}

\noindent We define $L_{B_i,f}$ and $L_{C_i,f}$ analogously to $L_{A,f}$. Note that the operators $L_{B_i,f}$ are the same as the operators $L_f$ acting on $F_\theta(\Sigma_{B_i})$ as given in Definition $2.2$.\\ \indent
We want to understand the spectral properties of the operators $L_{C_i,sf}$ for $|s|$ small. We analyse the operators in the case that $s=0$ and then use perturbation theory to obtain our desired result. Suppose that $\lambda$ is the exponential growth rate of $\#W_n$. It is well known that for each $i$, $L_{B_i,0}$ has the same simple maximal eigenvalues as $B_i$. These maximal eigenvalues have modulus $\lambda$ since the $B_i$ are maximal components. From our discussion in Section $2$, $\lambda$ is equal to $e^{h}$ where $h$ denotes the topological entropy of the system $(\Sigma_{B_i},\sigma)$. We want to show that $L_{C_i,0}$ has essentially the same spectrum as $L_{B_i,0}$.

\begin{lemma}
Suppose each $B_i$ has cyclic period $p_i$. Then, the operators $L_{C_i,0}$ are quasicompact, have spectra that consist of $p_i$ finite multiplicity maximal eigenvalues at $e^{2\pi i k /p_i} e^{h}$ for $k = 0,1,...,p_i-1$. The rest of the spectrum is contained in the disk $\{z: |z|<e^{h}-\delta\}$ for some $\delta>0$.
\end{lemma}

\begin{proof}
The proof is basically an application of Lemma $2$ from \cite{oc}. Quasicompactness of the operators follows immediately. By relabeling the columns of $C_i$, we can rewrite each $C_i$ in the form

$$C_i = \begin{pmatrix} 
C_{1,1} & 0 & \dots & 0  \\
C_{2,1} & C_{2,2} & \dots & 0\\
\vdots & \vdots & \ddots & \vdots\\
C_{m,1} & C_{m,2} & \dots & C_{m,m}
\end{pmatrix} $$
where the $C_{j,j}$ correspond to irreducible components of $\mathcal{G}$. By construction all maximal components have corresponding matrix $0$ except for the matrix corresponding to $B_i$.
Let
$$ P=
\begin{pmatrix} 
C_{1,1} &0 &0   &\dots & 0  \\
0 &C_{2,2} &0 &\dots & 0\\
0 &0 &C_{3,3} &\dots &0\\
\vdots & \vdots &\vdots & \ddots & \vdots\\
0 & 0  &0 &\dots & C_{m,m}
\end{pmatrix} $$
Lemma $2$ in \cite{oc} states that the operators $L_{C_i,0}$ and $L_{P,0}$ have the same isolated eigenvalues. It is easy to see that the spectrum of $L_{P,0}$ consists of $p_i$ finite multiplicity eigenvalues, $e^{2\pi i k/p_i} e^h$ for $k = 0,...,p_i-1$ and the rest of the spectrum is contained in $\{z:|z| <e^{h}-\delta\}$ for some $\delta>0$. Quasicompactness of the $L_{C_i,0}$ now implies the result.
\end{proof}

One can check that the finite multiplicity eigenvalues from the above lemma are in fact simple. Let $B_i^\ast$ denote the matrices that describes the subgraph of $\mathcal{G}$ that contains the vertices in $B_i$, the $0$ vertex and all edges between these vertices that are allowed by $A$. There are a few steps in showing that the eigenvalues in the above lemma are simple. We show that each of the following statements can, in some sense, be deduced from the previous one.

\begin{enumerate}
\item If $B_i$ is aperiodic then the maximal eigenvalue for $L_{B_i^\ast,0}$ is simple.
\item If $B_i$ is irreducible then the maximal eigenvalues for $L_{B_i^\ast,0}$ are simple.
\item If $B_i$ is irreducible then the maximal eigenvalues for $L_{C_i,0}$ are simple.
\end{enumerate}
Statement $(1)$ in the above is well known \cite{GR}, \cite{oc}. We will show that $(2)$ and then $(3)$ also hold.

\begin{proof} [Proof of (2)]

 \noindent Suppose that $B_i$ is irreducible. Recall that there exists $p_i$, the period of $B_i$, such that
$\Sigma_{B_i}$ has $p_i$-cyclic decomposition
$$\Sigma_{B_i} = \bigsqcup_{k=0}^{p_i-1} \Sigma_{B_i^k}.$$ 
$L_{B_i,0}$ has spectrum containing maximal eigenvalues at $e^{2\pi ik/p_i} e^h$ for $k = 0,1,...,p_i-1$. The rest of the spectrum is contained in a disk of radius strictly smaller than $e^h$.\\ \indent
The $p_i$th iterates of the transfer operators $L_{B_i,0}^{p_i}$ act as the direct sum of operators $L_{B_i^k,0}^{p_i}$ for $k = 0,...,p_i-1$ each acting on $F_\theta(\Sigma_{B_i^k})$ respectively. The analogous statement is true for the $L_{B_i^\ast,0}^{p_i}$. The following notation expresses this,
$$L_{B_i^\ast,0}^{p_i} = \left( L^{p_i}_{{B_{i,0}^\ast},0},L^{p_i}_{{B_{i,1}^\ast},0},...,L^{p_i}_{{B_{i,p_i-1}^\ast},0}\right),$$
$$L_{B_i,0}^{p_i} = \left( L_{B_i^0,0}^{p_i},L_{B_i^1,0}^{p_i},...,L_{B_i^{p-1},0}^{p_i}\right).$$
Here, each $B_{i,k}^\ast$ corresponds to $B_i^k$ with the $0$ vertex (and all edges to the $0$ vertex) added back in. We will continue to use the above notation through out the rest of this work.\\ \indent
Each $(\Sigma_{B^\ast_{i,k}},\sigma^{p_i})$ is a subshift of finite type of the same form as the aperiodic case from $(1)$. We know that, for each $k$, $L_{B_i^k,0}^{p_i}$ has simple maximal eigenvalue $e^{p_ih}$ and hence $L_{B^\ast_{i,k},0}^{p_i}$ does also. From the definition of $L_{B_i^\ast,0}$ it is easy to see that the spectrum of $L_{B_i^\ast,0}$ consists of simple maximal eigenvalues at $e^{2\pi ik/p_i} e^h$ for $k = 0,1,...,p_i-1$ and the rest of the spectrum is contained in the disk $\{z: |z|<e^{h}-\delta\}$ for some $\delta>0$. This concludes the proof.
\end{proof}

Using $(2)$ we can now prove that $(3)$ holds.

\begin{proof} [Proof of  (3) ]

\noindent Suppose that $B_i$ is irreducible and that $g \in F_\theta(\Sigma_{B_i^\ast})$ is the eigenfunction for the eigenvalue $e^{2\pi i k/p_i} e^h$ for $L_{B_i^\ast,0}$. Let $h$ be an eigenfunction corresponding to the eigenvalue $e^{2\pi i k/p_i} e^h$ for $L_{C_i,0}$. Suppose there exists $x \in \Sigma_{C_i}$ such that $x_0$ does not belong to $B_i$ but there exists a path from $x_0$ into $B_i$. Then,
\begin{align*}
e^{p_i hn}|h(x)| &=|L^{p_in}_{C_i,0} h(x)|\\
&\le \sum_{\substack{\sigma^{p_in}(y)=x \\ y \in \Sigma_{C_i} \backslash \{\dot{0}\}}}|h(y)|\\
&= \sum_{\substack{\sigma^{p_in}(y)=x \\ y \in \Sigma_{C_i} \backslash \{\dot{0}\} : \text{$y_0,...,y_{p_in-1}$ are not in $B_i$}}}\hspace{-1cm}|h(y)|.
\end{align*}
However, the growth of the number of length $n$ paths in $\mathcal{G},$ starting at $\ast$, that do not enter a maximal component, is $o(e^{hn})$. This implies that
$$e^{p_ihn} |h(x)| = o(e^{p_i h n}),$$
which forces $h(x)=0$. Hence $h$ is zero on 
$$S:= \{x \in \Sigma_{C_i} : \text{$x_0$ is not in $B_i$ and there exists a path from $x_0$ into $B_i$ in $\mathcal{G}$}\}.$$ 
We deduce that $h|_{\Sigma_{B_i^\ast}}$ is an eigenfunction for $L_{B_i^\ast,0}.$ Now, suppose $L_{C_i,0}$ has another eigenfunction for the eigenvalue $e^{2\pi i k/p_i}e^h$. Then, by taking a linear combination of $h$ and this new eigenfunction, we can assume that there exists a non-zero eigenfunction for $L_{B_i^\ast,0}$ that is zero on the set
$$\{x \in \Sigma_{C_i}: \text{there exists a path from $x_0$ into $B_i$ in $\mathcal{G}$}\}.$$
However, by taking $x$ such that $h(x) \neq 0$ and running the same growth argument as before, we see that any such eigenfunction can not exists. Hence $L_{C_i,0}$ has algebraically simple eigenvalues at $e^{2 \pi i k/p_i} e^h$ for $k = 0,...,p_i-1$.\\ \indent
To see geometric simplicity a similar argument can be applied. Suppose $L_{C_i,0}$ has Jordan chain
\begin{align*}
&g_{n-1}\\
&g_{n-2} = \left(L_{C_i,0} - e^{2\pi i k/p_i} e^{h}\right) g_{n-1}\\
&\vdots\\
&g = \left(L_{C_i,0} - e^{2\pi i k/p_i}e^{h}\right) g_1,
\end{align*}
for $n \ge 3$.
Then we see that there exists bounded linear operators $P_j(n)$ such that for each $j$ and $n \in \mathbb{Z}_{\ge 0}$,
$$L_{C_i,0}^{p_in}g_j = e^{np_i h} g_j + P_j(n)g_{j-1}.$$
By the same growth argument as before, if $g_{j-1}$ is $0$ on the set $S$, then $g_{j}$ is also $0$ on $S$. Hence, by induction, all the $g_j$ are $0$ on $S$.  This implies that $g,g_1,...,g_{n-1}$ restricts to a Jordan chain for $L_{B_i^\ast,0}$ which in turn implies $g|_{B_i^\ast} = 0$, a contradiction. This concludes the proof.
\end{proof}
In summary we have shown.
\begin{proposition}
Suppose each $B_i$ has cyclic period $p_i$. Then, there exists $\delta >0$ such that the operators $L_{C_i,0}$ have spectra that consist of $p_i$ simple maximal eigenvalues at $e^{2\pi i k /p_i} e^{h}$ for $k = 0,1,...,p_i-1$ and the rest of the spectrum is contained in $\{z: |z|<e^{h}-\delta\}$.
\end{proposition}

We now study the perturbed operators $L_{C_i,sf}$. We require the following result from perturbation theory.

\begin{proposition} \cite[Theorem 6.17]{k}
Let $B(V)$ denote the Banach algebra of bounded linear operators on a Banach space. Suppose $L_0$ has a simple isolated eigenvalue $\rho(L_0)$ with corresponding eigenvector $v(L_0)$. Then, for any $\epsilon >0$, there is $\delta>0$ such that if $\|L - L_0\| < \delta$ then $L$ has a simple isolated eigenvalue $\rho(L)$ with corresponding eigenvector $v(L)$. Moreover

\begin{itemize}
\item the maps $L \mapsto \rho(L)$ and $L \mapsto v(L)$ are analytic for $\|L-L_0\| < \delta$,
\item if $\|L-L_0\| < \delta$, then $|\rho(L) - \rho(L_0)| < \epsilon$ and the part of the spectrum of $L$ that does not include $\rho(L)$ is contained in $\{z \in \mathbb{C} : |z-\rho(L_0)| > \epsilon\}$.
\end{itemize}
\end{proposition}

By Proposition $4.4$, upper semi-continuity of the spectrum and Proposition $4.5$, for all sufficiently small (complex) $s$, $L_{C_i,sf}$ has $p_i$ simple maximal eigenvalues and exhibits a spectral gap to the rest of the spectrum. This gap is uniform for $s$ in a small neighbourhood of the origin. Our aim is to show that, as we perturb $L_{C_i,0}$, these simple maximal eigenvalues vary in the same way. Specifically, we want to show that for small $s$, $L_{C_i,sf}$ has $p_i$ simple maximal eigenvalues of the form $\lambda_s e^{2\pi i k/p_i}$ for $k=0,...,p_i-1$, where $s \mapsto \lambda_s$ is analytic.\\ \indent
By Lemma $2$ in \cite{oc}, for sufficiently small $s$, the simple maximal eigenvalues of $L_{C_i,sf}$ are those of $L_{B_i,sf}$. Hence it suffices to study small perturbations of $L_{B_i,0}$. Suppose $\Sigma_{B_i}$ has cyclic decomposition $\sqcup_{k=0}^{p_i-1} \Sigma_{B_i^k}$ as before.\\ \indent
We consider the $p_i\ $th iterate of $L_{B_i,0}$,
$$L_{B_i,sf}^{p_i} = \left( L_{B_i^0,sf}^{p_i},L_{B_i^1,sf}^{p_i},...,L_{B_i^{p_i-1},sf}^{p_i}\right).$$

\noindent The systems $(\Sigma_{B_i^k}, \sigma^{p_i})$ are aperiodic subshifts and $L_{B_i^k,sf}^{p_i}$ acts as $L_{B_i^k,sf^{p_i}}$ on this system. Define $\tilde{f}^k: \Sigma_{B_i^k} \to \mathbb{R}$ by $\tilde{f}^k(x) = f^{p_i}(x)$. We can choose $\epsilon>0$ such that for $|s|<\epsilon$ each of the $L_{B_i^k,sf^{p_i}}$ have a simple maximal eigenvalue $e^{P(s\tilde{f}^k)}$ and exhibit a spectral gap to the rest of the spectrum. Fix $|s|<\epsilon$. We deduce that the spectrum of $ L_{B_i,sf}^{p_i}$ consists of a finite multiplicity maximal eigenvalue $\lambda :=e^{P(s\tilde{f}^l)}$ for some $l\in \{0,...,p_i-1\}$ and the rest of the spectrum is contained in a disk, centered at the origin, of radius strictly less than $|\lambda|$. It is easy to see that if $x$ is in the spectrum of $L_{B_i,sf}^{p_i}$, then one of the $p_i$th roots of $x$ must be in the spectrum of $L_{B_i,sf}$. Furthermore, each element in the spectrum of $L_{B_i,sf}$ is the $p_i$th root of an element in the spectrum of $L_{B_i,sf}^{p_i}$. By quasicompactness $L_{B_i, sf}$ has an eigenvalue that is a $p_i$th root of $\lambda$. Suppose $g_0$ is the associated eigenfunction. Note that $g_0$ restricted to $\Sigma_{B_i^k}$ is an eigenfunction for each $k$ satisfying $L_{B_i^k,sf}^{p_i}g_0|_{B_i^k} = \lambda g_0|_{B_i^k}$. It follows from the definition of the transfer operator, that for each $k$, $g_0|_{B_i^k}$ is not identically zero (otherwise $g_0$ would be identically zero). We deduce that for all $s$ sufficiently small, the eigenvalues $e^{P( s\tilde{f}^k)}$ agree for all $k$. It follows that for all $s$ sufficiently small, the spectrum of $L_{B_i,sf}$ consists of $p_i$ simple maximal eigenvalues of the form $e^{2\pi i k/p_i} e^{P(s\tilde{f}^0)/p_i}$ for $k = 0,1,...,p_i-1$ and the rest of the spectrum is contained in a disk of radius strictly less than the modulus of $e^{P(s\tilde{f}^0)/p_i} - \delta$, for some $\delta>0$. To simplify notation we write $P_i(sf)$ to denote $P(s\tilde{f}^0)/p_i$ . To summarise, we have shown the following.

\begin{proposition}
There exists $\epsilon, \delta >0$ such that for all $|s|<\epsilon$, $L_{C_i,sf}$ has $p_i$ simple maximal eigenvalues $e^{2 \pi ik/p_i}e^{P_i(sf)}$ for $k = 0,...,p_i-1$, these are contained in the $\delta$ neighbourhood of $\{e^{2 \pi ik/p_i}e^h : k = 0,...,p_i-1\}$ and the rest of the spectrum is contained in the disk $|z|< e^{h}-2\delta$.
\end{proposition}

Let $\epsilon$ be as in the above proposition. We use $B(F_\theta(\Sigma_A))$ to denote the Banach algebra of bounded linear operators over $\Sigma_A$. Results from analytic perturbation theory (see Theorem 6.17 in \cite{k}) imply that there exist analytic projection valued functions $Q_{i,k}: \{s \in \mathbb{C}: |s| < \epsilon\} \to B(F_\theta(\Sigma_A))$ such that $Q_{i,k}(s)$ projects a function in $F_\theta(\Sigma_{C_i})$ to the one-dimensional eigenspace associated to the simple maximal eigenvalue $e^{2\pi i k/p_i} e^{P_i(sf)}$ of the operator $L_{C_i,sf}$.

\subsection{Comparing the Derivatives of Pressure}
We now want to show that, as we perturb the operators $L_{C_i,0}$, the simple maximal eigenvalues from Proposition $4.6$ vary in a similar way. Specifically, we show that the quantities 
$$\Lambda_i := \frac{dP_i(sf)}{ds}  \Big|_{s=0} \text{ and } \sigma^2_i:= \frac{d^2P_i(sf)}{ds^2} \Big|_{s=0}$$
are independent of the maximal component $B_i$.\\ \indent

To show that these quantities agree across components, we appeal to the work of Calegari and Fujiwara. We will use the argument presented in \cite{calnotes} and \cite{cf}. To apply this argument, we need the following technical lemma.
\begin{lemma}
Suppose $r = (r_k)_{k=0}^\infty \in \Sigma_A$ with $r_0 = \ast$. Write $\tilde{r}_k \in G$ to denote the group element corresponding to the path $(\ast,r_1,...,r_{k-1},\dot{0})$ in $\mathcal{G}$ under the bijection from Definition $2.5$. Then,
$$f^n(\sigma^k(r)) = \varphi(\tilde{r}_{n+k}) - \varphi(\tilde{r}_k) + O(1),$$
where the above error term constant is independent of $r,k$ and $n$.
\end{lemma}
\begin{proof}
Given $n,k \in \mathbb{Z}_{\ge 0}$ and $r \in \Sigma_A$, define $s_1,s_2, s_3 \in \Sigma_A$ by
$$ s_1 = (\ast,r_1,...,r_{k-1},\dot{0}), \hspace{.4cm} s_2 = (\ast,r_1,...,r_{k+n-1},\dot{0}), \hspace{.4cm} s_3 = (r_k,r_{k+1},...,r_{k+n-1},\dot{0}).$$
Then, by the H\" older property of $f$, there exists $C>0$ independent of $n,k$ and $r$, such that
$$|f^n(\sigma^k(r))-f^n(s_3)|\le C.$$
Then, note that $f^n(s_3) + f^k(s_2) = f^{n+k}(s_2)$ and also that there exists $C'>0$ independent of $n,k$ and $r$, such that
 $$|f^k(s_2)-f^k(s_1)|\le C'.$$
 Finally, by Condition $(1)$,
 $$f^k(s_1) = \varphi(\tilde{r}_k) \hspace{1mm} \text{and} \hspace{1mm} f^{k+n}(s_2)= \varphi(\tilde{r}_{k+n})$$
 and so
 \begin{align*}
 f^n(\sigma^k(r)) &= f^n(s_3) + O(1) \\
 &=f^{n+k}(s_2) - f^k(s_2) + O(1)\\
 &=\varphi(\tilde{r}_{n+k}) - \varphi(\tilde{r}_k) + O(1),
\end{align*}
where the implied constant term is independent of $n,k$ and $r$.
\end{proof}
The main result of this section is the following. Recall that $\nu$ denotes the Patterson-Sullivan measure on $\partial G$.
\begin{proposition}
The quantities, $\Lambda_i$ and $\sigma^2_i$ do not depend on $i=1,...,m$.
\end{proposition}
\begin{proof}
\indent Let $Y \subset \Sigma_A$ consist of all elements in $\Sigma_A$ that begin with the $\ast$ vertex and do not end in infinitely many zeros. Let $h : Y \to \partial G$ denote the natural map arising from the labeling given in Definition $2.5$. In \cite{cf} Calegari and Fujiwara construct a measure $\mu$ on $\Sigma_A$ such that
\begin{equation}
\mu = \lim_{m\to \infty} \frac{1}{m} \sum_{k=0}^m \sigma_\ast^k\widehat{\nu},
\end{equation}
where $\widehat{\nu}$ is the unique measure on $Y$ that pushes forward under $h$ to the measure
$$\nu' = \lim_{n\to\infty} \frac{1}{n}\sum_{|g|\le n} \lambda^{-|g|} \delta_g.$$ Here, $\sigma_\ast$ denotes the push forward, i.e. $\sigma_\ast^k\widehat{\nu}( \cdot ) = \widehat{\nu}(\sigma^{-k} \cdot )$ for all $k \in \mathbb{Z}_{\ge 0}$.  The measure $\nu'$ is, up to scaling, the Patterson-Sullivan measure $\nu$.\\ \indent
Let $\mu_i$ denote the measure $\mu$ restricted to $\Sigma_{B_i}$. By comparing Calegari and Fujiwara's construction of $\mu$ to Parry's construction of the measure of maximal entropy, $\widehat{\mu}_i$ for $\Sigma_{B_i}$ \cite{parry}, we see that $\mu_i$ and $\widehat{\mu}_i$ agree up to scaling. We can assume that $\mu_i$ is scaled to be a probability measure, so that it agrees with $\widehat{\mu}_i$.\\ \indent
There is a stronger relationship between $\mu$ and $\widehat{\nu}$ than that given in $(4.1)$. We can deduce from Lemma $4.19$ and Lemma $4.22$ in \cite{cf} that if $\mu_i(E) >0$ for some set $E \subset \Sigma_{B_i}$, then there exists $k\in\mathbb{Z}_{\ge0}$ such that $\sigma_\ast^k\widehat{\nu}(E)>0$. This property is the key ingredient that allows us to compare the $\widehat{\nu}$ measure with the $\mu$ measure.\\ \indent
The measure $\mu_i$ is ergodic with respect to $\sigma$ on $\Sigma_{B_i}$ and by the ergodic theorem, if $g \in L^1(\Sigma_{B_i}, \mu_i)$ then
$$\frac{1}{m} g^m(z) \rightarrow \int g \ d \mu_i,$$
as $m\to\infty$, for $\mu_i$ a.e $z \in \Sigma_{B_i}$. We define
$$F(n,x) = \left\{ r \in \Sigma_{B_i} : \frac{f^n(r) - \Lambda_i n}{\sqrt{n}} \le x\right\}$$
and
$$\mu(z,m) = \frac{1}{m} \sum_{k=0}^m \delta_{\sigma^k z}.$$
Throughout the following it is helpful to keep the following expression in mind,
$$ \int \mathbbm{1}_{F(n,x)} \ d \mu(z,m)  = \frac{1}{m}\#\left\{0\le j \le m: \frac{f^n(\sigma^j(z)) - \Lambda_i n}{\sqrt{n}} \le x \right\}$$
where $\mathbbm{1}_{F(n,x)}$ denotes the indicator function for $F(n,x)$. To simplify our notation in the following, if $\sigma_i^2=0$, then we take
$$\frac{1}{\sqrt{2 \pi} \sigma_i} \int_{-\infty}^x e^{-t^2/2\sigma_i^2} \ dt$$
to be the Heaviside function. The central limit theorem for subshifts of finite type \cite{cp} implies that there exists a set $N_i \subset \Sigma_{B_i}$ with $\mu_i(N_i)=1$, such that for all $x \in \mathbb{R}$ and $z \in N_i$,
\begin{align*}
\lim_{n\to \infty} \lim_{m\to \infty} \int \mathbbm{1}_{F(n,x)} \ d \mu(z,m) &= \lim_{n\to \infty} \mu_i (F(n,x))\\
&= \frac{1}{\sqrt{2 \pi} \sigma_i} \int_{-\infty}^x e^{-t^2/2\sigma_i^2} \ dt.
\end{align*}
We note that if $z\in\Sigma_{A}$ satisfies the above convergence, then any pre-image $y \in \sigma^{-1}(z)$ also satisfies the above convergence. Also, from the above discussion, there exists $k\in\mathbb{Z}_{\ge0}$ such that $\sigma_\ast^k\widehat{\nu}(N_i) >0$. Combining these observations implies that there exists a set $E_i \subset Y$ of positive $\widehat{\nu}$ measure and for $x \in \mathbb{R}$,
\begin{equation}
\lim_{n\to \infty} \lim_{m\to \infty} \int \mathbbm{1}_{F(n,x)} \ d \mu(y,m) = \frac{1}{\sqrt{2 \pi} \sigma_i} \int_{-\infty}^x e^{-t^2/2\sigma_i^2} dt,
\end{equation}
when $y \in E_i$. Hence, for each $i=1,...,m$, $h(E_i) \subset \partial G$ has positive $\nu$ measure.\\ \indent
We define the set $S_i \subset \partial G$ to be the collection of elements in $\partial G$ that have a corresponding infinite geodesic ray $\gamma$ such that for all $x \in \mathbb{R}$,
\begin{align*}
\lim_{n \to \infty} \limsup_{m\to \infty} \frac{1}{m} \#\left\{0\le j \le m: \frac{\varphi(\gamma_{j+n}) - \varphi(\gamma_j) - \Lambda_i n}{\sqrt{n}} \le x \right\} &= \\
&\hspace{-1cm} \frac{1}{\sqrt{2 \pi} \sigma_i} \int_{-\infty}^x \hspace{-2mm} e^{-t^2/2\sigma_i^2} dt.
\end{align*}
Since $\varphi$ is Lipschitz in the left and right word metric, if $\gamma_1, \gamma_2$ are two geodesic rays with the same end point in $\partial G$, then $\gamma_1$ satisfies the above convergence if and only if $\gamma_2$ does. Further, as $\varphi$ is Lipschitz in the right word metric $S_i$ is $G$ invariant. See Lemma $4.3$ in \cite{cf} for a more detailed explanation of these last two points. \\ \indent
This $G$ invariance implies that, by the ergodicity of the action of $G$ on $\partial G$ with respect to $\nu$, $\nu(S_i)$ either has full measure or zero measure. 
However, Lemma $4.7$ and expression $(4.2)$ imply that $h(E_i) \subset S_i$. To see this note that for $y \in E_i$,
$$\frac{1}{m}\#\left\{0\le j \le m: \frac{\varphi(h(y)_{j+n}) - \varphi(h(y)_j) - \Lambda_i n}{\sqrt{n}} \le x \right\}$$
is equal to 
$$\frac{1}{m}\#\left\{0\le j \le m: \frac{f^n(\sigma^j(y)) + O(1) - \Lambda_i n}{\sqrt{n}} \le x \right\} = \int \mathbbm{1}_{F(n,x+O(n^{-1/2}))} \ d \mu(y,m),$$
where the above error term arises from the application of Lemma $4.7$. This error term does not affect the convergence exhibited in $(4.2)$ and we deduce that $h(E_i) \subset S_i$. Since $\nu(h(E_i))>0$, $S_i$ has full measure. It follows that the $S_i$ coincide and hence that $\Lambda_i$ and $\sigma^2_i$ do not depend on $i=1,...,m$ as required.
\end{proof}
From now on, we use the notation
$$\Lambda_\varphi : = \frac{d}{ds} P_i(sf)\Big|_{s=0} \text{ and } \sigma_\varphi^2 := \frac{d^2}{ds^2} P_i(sf)\Big|_{s=0},$$
for any $i=1,...,m$.\\ \indent
By the above discussion $\Lambda_\varphi$ and $\sigma_\varphi^2$ are well defined i.e. independent of the choice of maximal component. Computing the Taylor expansion of each $P_i(sf)$ in a neighbourhood of zero gives the following.
\begin{lemma}
There exists a neighbourhood $U$ of $0$ in $\mathbb{C}$ such that for $s \in U$ and for each $i=1,...,m$,
\begin{equation} \label{eq:3}
P_i(sf) = h + \Lambda_\varphi s + \sigma_\varphi^2s^2/2 + O(s^3)
\end{equation}
as $s\to 0$.
\end{lemma}

\section{Cohomology Conditions}
\subsection{Positive variance}
Throughout the following we use the notation established in the previous two sections. The aim of this section is to characterise the case that $\sigma_\varphi^2 = 0$. Let $B_i$ be a maximal component with cyclic decomposition
$$\Sigma_{B_i} = \bigsqcup_{j=0}^{p_i-1} \Sigma_{B_i^j}.$$
Let $B_i^G$ denote the elements in $G$ that can be realised as a word corresponding to a path contained in the component $B_i$. Specifically, let $\rho$ denote the labeling map from Definition $2.5$, then $B_i^G$ is the set,
$$\{ g \in G: g = \rho(e_0)\rho(e_1)...\rho(e_{n-1}) \text{ for some path with edges $e_0,...,e_{n-1}$ in $B_i$}\}.$$
Recall that for small $s$, the spectral radius of the operator $L_{C_i,sf}$ is given by the modulus of $e^{P_i(sf)}$. Furthermore, $P_i(sf)$ denotes the quantity $P(s\tilde{f}^0)/{p_i}$ where $\tilde{f}^0$ is the function $f^{p_i}$ restricted to $\Sigma_{B_i^0}$.
\begin{lemma} \label{cohomcond}
Suppose $\varphi$ satisfies Condition $(1)$ and Condition $(2)$ with associated potential $f:\Sigma_A\to \mathbb{R}$. Let $(f^{p_i})^n(x)$ denote $f^{p_i}(x)+f^{p_i}(\sigma^{p_i}(x)) + ...+f^{p_i}(\sigma^{p_i(n-1)}(x))$. Then, the following are equivalent
\begin{enumerate}
\item $\sigma_\varphi^2 =0$,
\item The function $f^{p_i}$ on $(\Sigma_{B_i^0},\sigma^{p_i})$ is cohomologous to a constant,
\item $\{(f^{p_i})^n(x)-np_i\Lambda_\varphi : x \in \Sigma_{B_i^0}, n\in \mathbb{Z}_{\ge 0}\}$ is bounded,
\item $\{(f^{p_i})^n(x) - np_i\Lambda_\varphi : x \in \Sigma_{B_i^j}, n\in \mathbb{Z}_{\ge 0}\}$ is bounded for $j = 0,1,...,p_i-1$,
\item $\{f^n(x) - n\Lambda_\varphi : x \in \Sigma_{B_i}, n\in \mathbb{Z}_{\ge 0}\}$ is bounded,
\item $\{\varphi(g) - |g|\Lambda_\varphi : g \in B_i^G\}$ is bounded,
\item $\{\varphi(g) - |g|\Lambda_\varphi : g \in G\}$ is bounded.
\end{enumerate}
\end{lemma}

\begin{proof}
$(1) \iff (2)$ This is a standard result. See \cite{pp}.\\
$(2) \iff (3)$ This is proved in \cite{GR}, see Lemma $2.3$.\\
$(3) \iff (4)$ This follows from the discussion leading up to Proposition $4.6$.\\ 
$(4) \iff (5)$ This is a simple exercise.\\
$(5) \iff (6)$ Given $g \in B_i^G$, we can view $g$ as a path contained in the component $B_i$. We can then extend this path on the left to a path that begins at the $\ast$ vertex and on the right so that it ends at the $0$ vertex. Furthermore, there exists $L \in \mathbb{Z}_{\ge 0}$ such that we can always extend a group element in this way by adding at most $L$ new vertices. This extended path corresponds to a group element $g' \in G$ and we have that, by Condition $(2)$,
$$\varphi(g) = \varphi(g') + O(1),$$
where the implied constant is independent of $g$ and $g'$. Then, using the embedding $i: G \to \Sigma_A$ we see that
$$\varphi(g) = f^{|g|}(\sigma^{|g'|-|g|}(i(g'))) + O(1),$$ where the implied constant is independent of $g$. Now choose any $x=(x_k)_{k=0}^\infty \in \Sigma_{B_i}$ for which $x_0,x_1,...,x_{|g|}$ describes the path related to $g$. Then, by the H\" older condition on $f$,
$$\varphi(g)=f^{|g|}(x) +O(1),$$ where the implied constant is independent of $g$ and our choice of $x$. This gives one of our desired implications. Running this argument backwards gives the other.\\
$(6) \iff (7)$ This is a consequence of hyperbolic groups being growth quasitight (see Definition 1.5 in \cite{gtt2}). By Lemma $4.6$ of \cite{gouz} there exists a finite set $M \subset G$ such that $MB_i^G M = G$ (see also Proposition 7.2 of \cite{gtt2}). The conclusion then follows easily from Condition (2).
\end{proof}
\begin{definition}
We say that $\varphi(\cdot) - \Lambda_\varphi |\cdot|$ is unbounded if $\left\{ \varphi(g) - |g|\Lambda_\varphi: g \in G\right\}$ is an unbounded subset of $\mathbb{R}$. 
\end{definition}

\begin{remark}
Lemma $5.1$ characterises the degenerate case for Calegari and Fujiwara's central limit theorem \cite{cf}. This is because, as discussed earlier, the functions considered by Calegari and Fujiwara have an associated H\" older potential and the variance, $\sigma_\varphi^2$, associated to this potential agrees with the variance in Calegari and Fujiwara's central limit theorem.
\end{remark}

\subsection{Other positive variance conditions and consequences of our results}
We can use the positive variance conditions from Lemma \ref{cohomcond} to deduce combinatorial and geometric properties of functions satisfying Condition (1) and Condition (2). The remainder of this section is dedicated to this end.\\ \indent
We begin by defining the following set
$$U =\{ [\gamma] \in \partial G : \{\varphi(\gamma_n)-|\gamma_n|\Lambda_\varphi : n\in \mathbb{Z}_{\ge 0}\} \text{ is unbounded}\}.$$
This set is well defined because $\varphi$ is Lipschitz in the left word metric on $G$. Given $[\gamma],[\gamma'] \in \partial G$, there exists $C>0$ such that
$$|(\varphi(\gamma_n)-|\gamma_n|\Lambda_\varphi) - (\varphi(\gamma'_n)-|\gamma'_n|\Lambda_\varphi)| = |\varphi(\gamma_n)-\varphi(\gamma'_n)| + ||\gamma_n|-|\gamma'_n||\le  C d_L(\gamma_n,\gamma'_n).$$
If $[\gamma] = [\gamma']$, the right hand side of the above is bounded uniformly in $n$ and hence $\{\varphi(\gamma_n)-|\gamma_n|\Lambda_\varphi : n\in \mathbb{Z}_{\ge 0}\}$ is bounded if and only if $\{\varphi(\gamma'_n)-|\gamma'_n|\Lambda_\varphi : n\in \mathbb{Z}_{\ge 0}\}$ is bounded.
\begin{definition}
We say that $\varphi$ is unbounded on the boundary if $\nu(U)>0$.
\end{definition}

\begin{remark}
 As $\varphi$ is Lipschitz in the right word metric, $U$ is $G$-invariant. Therefore by the ergodicity of the action of $G$ on $\partial G$ with respect to $\nu$, $\nu(U)=0$ or $1$. Hence the above definition can be equivalently stated by changing $\nu(U)>0$ to $\nu(U)=1$.
\end{remark}

\begin{proposition}
A function $\varphi:G\to \mathbb{R}$ satisfying Condition $(1)$ and $(2)$ is unbounded on the boundary if and only if $\varphi(\cdot) - \Lambda_\varphi |\cdot|$ is unbounded.
\end{proposition}
\begin{proof}
It is clear that if $\varphi$ is unbounded on the boundary then $\varphi(\cdot) - \Lambda_\varphi|\cdot|$ is unbounded.\\ \indent
Conversely, suppose that $\varphi(\cdot) - \Lambda_\varphi |\cdot|$ is unbounded. Let $\nu, \widehat{\nu},\mu$ and $\mu_i$ denote the measures defined in the proof of Proposition $4.8$ and let $h: \Sigma_A \to \partial G$ denote the map defined in this proposition. Since  $\varphi(\cdot) - \Lambda_\varphi |\cdot|$ is unbounded, $f$ satisfies a non-degenerate central limit theorem on a maximal component $B_i$ with respect to the measure of maximal entropy $\mu_i$ on that component, i.e., for $y \in \mathbb{R}$,
$$\lim_{n\to\infty}\mu_i(G(n,y)) = \frac{1}{\sqrt{2 \pi} \sigma_i} \int_{y}^\infty e^{-t^2/2\sigma_i^2} \ dt,$$
where
$$G(n,y) = \left\{ x \in \Sigma_{B_i} : \frac{f^n(x)-\Lambda_i n}{\sqrt{n}} \ge y\right\} \text{ and } \sigma^2_i >0.$$
Hence for any $y \in \mathbb{R}$,
\begin{align*}
\mu_i(\limsup_{n\to\infty} G(n,y)) &= \mu_i\left(\bigcap_{n\ge1}\bigcup_{j\ge n} G(j,y)\right)\\
&\ge \limsup_{n\to\infty} \mu_i(G(n,y)) > 0.
\end{align*}
Now fix $y >0$ and note that
$$\mu\left\{x \in \Sigma_{A} : \{f^n(x) - n\Lambda_i : n\in \mathbb{Z}_{\ge0}\} \text{ is unbounded }\right\} \ge 
\mu_i(\limsup_{n\to\infty} G(n,y)) >0.$$
As in the proof of Proposition $4.8$, the relationship between $\widehat{\nu}$ and $\mu$ implies that
$$\widehat{\nu}\left\{x \in \Sigma_A: x_0=\ast \text{ and } \{f^n(x) - n\Lambda_i : n\in \mathbb{Z}_{\ge0}\} \text{ is unbounded }\right\} >0.$$
Then, by Condition $(1)$ and the H\" older properties of $f$, for $x \in \Sigma_A$,
$$f^n(x) - n\Lambda_i  = \varphi(g) - |g|\Lambda_i + O(1),$$
where $g$ is the unique group element such that $i(g) = (\ast,x_0,...,x_{n-1},0,0,...)$. The implied constant in the above is independent of $x$. Lastly, since $\widehat{\nu}$ pushes forward under $h: \Sigma_A \to \partial G$ to $\nu$ on $\partial G$,
$$\nu\left\{ [\gamma] \in \partial G : \{\varphi(\gamma_n) - |\gamma_n|\Lambda_i : n\in\mathbb{Z}_{\ge0}\} \text{ is unbounded}\right\} >0$$
and $\varphi$ is unbounded on the boundary.
\end{proof}

The following is a combinatorial condition that is equivalent to $\varphi(\cdot) - \Lambda_\varphi |\cdot|$ being unbounded.

\begin{definition}
We say that $\varphi$ is unbounded on a thick domain, if whenever a subset, $H \subset G$ has the property that
$$\{ \varphi(g) - |g|\Lambda_\varphi : g \in H\}$$ 
is bounded, then the asymptotic density of $H$ with respect to $W_n$ is zero, i.e.
$$\lim_{n \to \infty} \ \frac{\#(W_n \cap H  )}{\#W_n} =0. $$
\end{definition}

\begin{lemma}
A function satisfying Condition $(1)$ and Condition $(2)$ is unbounded on a thick domain if and only if $\varphi(\cdot) - \Lambda_\varphi |\cdot|$ is unbounded.
\end{lemma}

\begin{proof}
It is clear that if $\varphi$ is unbounded on a thick domain, then $\varphi(\cdot) - \Lambda_\varphi|\cdot|$ is unbounded.\\
\indent Conversely, suppose that $\varphi(\cdot) - \Lambda_\varphi |\cdot|$ is unbounded and that $H \subset G$ is such that $\{\varphi(g) - \Lambda_\varphi |g| : g\in H\}$ is bounded. There then exists real $M>0$ such that,
$$\#(W_n \cap H )  \le \#\left\{g \in W_n : \frac{\varphi(g) - n\Lambda_\varphi}{\sqrt{n}} \in \left[\frac{-M}{\sqrt{n}}, \frac{M}{\sqrt{n}}\right]\right\}$$
for all $n \ge 1$. Applying Theorem \ref{CLT} then gives that, as $n \to\infty$,
$$  \frac{\#(W_n \cap H  )}{\#W_n} = \frac{1}{\sqrt{2 \pi} \sigma_\varphi} \int_{-M n^{-1/2}}^{Mn^{-1/2}}e^{-t^2/2\sigma_\varphi^2} \ dt + O(n^{-1/2}) = O(n^{-1/2}).$$
\end{proof}
\begin{remark}
\noindent  The proof of Lemma $5.8$ shows that we can replace the limit in Definition $5.7$ with a limit infimum without affecting the class of functions that are unbounded on a thick domain.
\end{remark}

We will now provide a class of functions that satisfy our central limit theorem with positive variance.


\begin{lemma}\label{homomcond}
If $\varphi:G\to \mathbb{R}$ is a non-trivial group homomorphism or an unbounded quasimorphism satisfying Condition (1), then $\sigma_\varphi^2 >0$.
\end{lemma}

\begin{proof}
From Theorem \ref{averaging} and the equalities $|g|=|g^{-1}|$ and $\varphi(g)=-\varphi(g^{-1})$ that hold for all $g \in G$, we see that $\Lambda_\varphi =0$. The result follows.
\end{proof}

Combining Propositions $5.6$ and $5.8$ gives the following result.
\begin{corollary}\label{homomcoro}
Suppose $G$ is a non-elementary hyperbolic group and $\varphi: G \to \mathbb{R}$ satisfies Condition (1), Condition (2) and that $\varphi(\cdot) - \Lambda_\varphi |\cdot|$ is unbounded. Then the subset of $\partial G$ consisting of (equivalence classes of) geodesic rays along which $\varphi(\cdot) - \Lambda_\varphi |\cdot|$ is unbounded, has full Patterson-Sullivan measure. Furthermore, if $\varphi(\cdot) - \Lambda_\varphi |\cdot|$ is bounded on $H \subset G$, then
$$\frac{\#(W_n \cap H)}{\#W_n} = O\left(\frac{1}{\sqrt{n}}\right)$$
as $n\to\infty$.
\end{corollary}

\section{Counting with Transfer Operators and a Simplification}
We are nearly ready to prove our results. Before doing so, in this section, we explain how we will make use of transfer operators in our proofs. We also establish the notation that we will use and make an observation that will allow us to simplify our analysis.\\ \indent
As mentioned previously, to prove our results, we need an understanding of the sums
$$ \sum_{g \in W_n} e^{s\varphi(g)},$$
for $s\in \mathbb{C}$ with $|s|$ small, as $n\to\infty$. We now show how to express this quantity in terms of transfer operators. This expression highlights the link between the geometrical setting of $\varphi$ on $G$ and the dynamical setting of $f$ on $\Sigma_A$. Let $\chi$ denote the indicator function for the set $ \{ (x_n)_{n=0}^\infty \in \Sigma_A : x_0 = \ast\}$.

\begin{lemma}\label{express}
There exists $\epsilon, \delta >0$ such that for $|s| < \epsilon$, each $L_{C_i,sf}$ has spectrum as described in Proposition $4.6$ and
$$ \sum_{g \in W_n} e^{s\varphi(g)} = \sum_{i=1}^m L^n_{C_i,sf}\chi(\dot{0}) + O\left(e^{n(h-\delta)}\right),$$
where the implied constant is independent of $|s|<\epsilon$.
\end{lemma}

\begin{proof}
Note that
\begin{equation}
 \sum_{g \in W_n} e^{s\varphi(g)} = \sum_{z} e^{s f^n(z)}.
\end{equation}
where the second sum is taken over $\{z \in \Sigma_A : \sigma^n(z)=\dot{0}, z_0 = \ast, z_{n-1}\neq \dot{0}\}$. Hence, the quantity
$$\sum_{i=1}^m L^n_{C_i,sf}\chi(\dot{0}),$$
expresses $(6.1)$ up to overcounting contributions from elements belonging to
\begin{align*}
\{ z \in \Sigma_A : \sigma^n(z)=\dot{0}, z_0&= \ast, z_{n-1}\neq 0 \text{ and the path corresponding}\\
&\hspace{1.5cm} \text{to $z$ does not enter a maximal component}\}.
\end{align*}
 Since the cardinality of this set is $O(e^{n(h-\nu)})$ for some $\nu > 0$ and $f$ is bounded, the result follows. 
\end{proof}

We now establish the notation that we will use throughout the remaining sections. Suppose $G$ is equipped with a generating set $S$. Suppose $G,S$ has associated directed graph $\mathcal{G}$ described by transition matrix $A$. Let $W_n$ denote the elements in $G$ of word length $n$ and let $\#W_n$ denote the cardinality of $W_n$. Let $B$, $B_i$ and $C_i$ for $i=1,...,m$ denote the matrices defined in Section $2$ and suppose that $\varphi:G \to \mathbb{R}$ is a function satisfying Condition $(1)$ and Condition $(2)$. Suppose $\varphi$ has associated potential $f \in F_\theta(\Sigma_A)$. Let $L_{C_i,sf}$ denote the transfer operators defined in Section $4$ and let $Q_{i,k}$ denote the projection valued operators defined after Proposition $4.6$. Denote by $Q_i$ the projection
$$Q_i = \sum_{k=0}^{p_i-1} Q_{i,k}.$$

\noindent Let $\Lambda_\varphi$ and $ \sigma_\varphi^2$ be the quantities related to $\varphi$ that were defined in Section $6$.\\ \indent
Throughout our proofs, we use the notation established above. The following lemma will allow us to simplify our analysis.

\begin{lemma}
Define $\gamma: G \to \mathbb{R}$ by $\gamma(g) =\varphi(g) - |g|\Lambda_\varphi$. Then $\gamma$ satisfies Condition $(1)$ and Condition $(2)$ and the potential related to $\gamma$ is $f - \Lambda_\varphi$. Furthermore
$$\Lambda_{\gamma} = 0 \hspace{2mm} \text{ and } \hspace{2mm} \sigma_{\gamma}^2 = \sigma_\varphi^2.$$
\end{lemma}

\begin{proof}
It is easy to check that the word length function $g \mapsto |g|$ satisfies Conditions $(1)$ and $(2)$ with related potential given by the constant function with value $1$. It follows that $\gamma$ also satisfies Conditions $(1)$ and $(2)$ with potential $f - \Lambda_\varphi$. Using the notation established in Section $6$, for any chosen maximal component with index $i$,
$$\Lambda_{\gamma} = \frac{d}{ds} P_i(s(f - \Lambda_\varphi)) \Big|_{s=0} \hspace{2mm} \text{ and } \hspace{2mm} \sigma_{\gamma}^2 = \frac{d^2}{ds^2} P_i(s(f - \Lambda_\varphi)) \Big|_{s=0}.$$
For real $s$ we have that, 
$$P_i(s(f - \Lambda_\varphi)) = P_i(sf) -s\Lambda_\varphi,$$
from which the remainder of the lemma easily follows.
\end{proof}

\noindent \textbf{Assumption:} The above lemma implies that, by swapping $\varphi$ to $\gamma$,  it suffices to prove Theorems \ref{averaging}, \ref{CLT} and $1.3$ under the assumption that $\Lambda_\varphi =0$. We assume this throughout the remaining sections.\\ \indent

We are now ready to move on to the proofs of our main results. We begin with the proof of Theorem \ref{averaging}.

\section{Averaging Theorem}

The aim of this section is to prove Theorem \ref{averaging}. Our proof is based around the analysis of the following generating function.
\begin{definition} Let
$$\eta(z,s) = \sum_{n=0}^\infty \frac{z^n}{n} \sum_{g \in W_n} e^{s \varphi(g)}.$$
\end{definition}
We want to study the domain of analyticity for $\eta$.
\begin{lemma}
We have that
$$\eta(z,s) = \sum_{n=0}^\infty \frac{z^n}{n} \sum_{i=1}^m L^n_{C_i,sf} \chi (\dot{0}) + \alpha(z,s),$$
for some function $\alpha(z,s)$ that is bi-analytic in $\{z: |z| < e^{-h + \delta}\} \times \{s: |s| < \epsilon\} $ for some $\epsilon, \delta >0$.
\end{lemma}

\begin{proof}
Let $\epsilon, \delta >0$ be as in Lemma $6.1$. Using Lemma $6.1$ we can write, for $|s|<\epsilon$, 
$$\sum_{g \in W_n} e^{s\varphi(g)} = \sum_{i=1}^m L^n_{C_i,sf}\chi(\dot{0}) + \omega_n(s),$$
where $\omega_n(s)$ is analytic in $|s|<\epsilon$ and $\omega_n(s) = O\left(e^{h(n-\delta)}\right)$. The implied constant is uniform in $|s|<\epsilon$. Define 
$$\alpha(z,s) = \sum_{n=0}^\infty \frac{z^n}{n} \ \omega_n(s) .$$
Clearly $\alpha$ satisfies the required identity for the lemma. Further, since the error term associated to $\omega_n$ is independent of $s$, for fixed $|s_0|<\epsilon$, $\alpha(z,s_0)$ is analytic in $\{z:|z|>e^{-h+\delta}\}$. Conversely, for fixed $|z_0| < e^{-\epsilon + \delta}$, $\alpha(z_0,s)$ is analytic in ${|s|<\epsilon}$. Hence, by Hartogs' Theorem (see Theorem $1.2.5$ in \cite{har}), $\alpha(z,s)$ satisfies the required analyticity condition.
\end{proof}

Let $\epsilon>0$ be as in Lemma $6.1$. By the Spectral Radius Theorem and Lemma $6.1$, there exists $\delta'>0$ such that
$$\left|\sum_{i=1}^m L^n_{C_i,sf}\chi(\dot{0})\right| = O\left(e^{n(h+\delta')}\right),$$
where the error term is independent of $|s|<\epsilon$. Lemma $4.6$ and an application of Hartogs' Theorem then implies that $\eta$ is bi-analytic in $\{|z| < e^{-h-\delta'}\}\times \{s:|s|< \epsilon\} $. Taking the derivative of $\eta$ with respect to $s$ at $s=0$ gives,
$$\frac{d}{ds} \eta(z,s) \bigg|_{s=0} = \sum_{n=0}^\infty \frac{z^n}{n} \sum_{g \in W_n} \varphi(g).$$

Let $\epsilon, \delta>0$ be as in Lemma $6.1$. Recall that we have analytic projection valued functions $Q_{i,k}$ for the simple maximal eigenvalues of the transfer operators $L_{C_i,sf}$. For $|s|<\epsilon$ , $Q_{i,k}(s)$ is the eigenprojection associated to the eigenvalue $e^{2\pi i k/p_i}e^{P_i(sf)}$ for $L_{C_i,sf}$. Using these projections we write
\begin{equation}
\sum_{i=1}^m L^n_{C_i,sf}\chi(\dot{0}) = \sum_{i=1}^m \sum_{k=1}^{p_i} e^{2 \pi ink/p_i} e^{nP_i(sf)} Q_{i,k}(s)\chi(\dot{0}) + O(e^{n(h-\delta)}),
\end{equation}
which is valid for $|s|<\epsilon$.\\ \indent
Using identity $(7.1)$, we can apply the same argument as in the proof of Lemma $7.2$, to the function
$$ \sum_{n=0}^\infty \frac{z^n}{n} \sum_{i=1}^m L^n_{C_i,sf} \chi (\dot{0}),$$
to deduce the following.
\begin{lemma}
\begin{equation}
\eta(z,s) =  \sum_{n=0}^\infty \frac{z^n}{n} \sum_{i=1}^m \sum_{k=1}^{p_i} e^{2 \pi ink/p_i} e^{nP_i(sf)} Q_{i,k}(s)\chi(\dot{0}) + \beta(z,s),
\end{equation}
for some $\beta(z,s)$ that is bi-analytic in $\{z: |z| < e^{-h+\epsilon}\} \times \{s:|s| < \delta\}$ for some $\epsilon, \delta>0$.
\end{lemma}
We then turn our attention to the double sum in $(7.2)$.
\begin{lemma}
Define
\begin{equation} 
\psi_n(s):=\sum_{i=1}^m \sum_{k=1}^{p_i} e^{2 \pi ink/p_i} e^{nP_i(sf)} Q_{i,k}(s)\chi(\dot{0}).
\end{equation}
Then each $\psi_n$ is analytic in a neighbourhood of $0$ and
$$ \psi_n'(0)=O\left(e^{nh}\right).$$
\end{lemma}
\begin{proof}
We recall that the projections $Q_{i,k}$ are analytic in a small neighbourhood of the origin. Hence the maps $s \mapsto Q_{i,k}(s)\chi(\dot{0})$ are analytic in a neighbourhood of the origin. Differentiating each $\psi_n$ and using the Taylor expansions for the pressure $(4.3)$ (recalling that $\Lambda_\varphi=0$), gives the required result.
\end{proof}
Taking the derivative of expression $(7.2)$ with respect to $s$ at $s=0$ and then rearranging, we obtain
\begin{equation}
 \sum_{n=0}^\infty z^n \left( -\frac{\psi_n'(0)}{n} + \sum_{g \in W_n} \frac{\varphi(g)}{n} \right) = \frac{d}{ds} \beta(z,s)\Big|_{s=0}.
\end{equation}
The domain of bi-analyticity for $\beta$ implies that the radius of convergence of the above series is strictly greater than $e^{-h}$.\\ \indent
We are now ready to prove our result.

\begin{proof} [Proof of Theorem \ref{averaging}]
\noindent Equation $(7.4)$ implies that
$$  \sum_{g \in W_n} \frac{\varphi(g)}{n} =  \frac{\psi_n'(0)}{n} + O(e^{n(h-\delta)})$$
for some $\delta >0$.\\ \indent
Dividing the above identity by $\#W_n$ and then applying Proposition $2.7$ and Lemma $7.4$ implies that
$$\frac{1}{\# W_n}\sum_{g \in W_n}\frac{\varphi(g)}{n} = O\left(\frac{1}{n}\right)$$
as required.
\end{proof}

\section{Central Limit Theorem}
We now move on to the proof of Theorem \ref{CLT}. Throughout this section, suppose that $\varphi$ is unbounded. By Lemma $5.1$ we have that $\sigma_\varphi^2 >0$. Recall that we want to study the convergence of the distributions
$$F_n(x) = \frac{1}{\# W_n} \# \left\{ g \in W_n : \frac{\varphi(g)}{\sqrt{n}} \le x \right\}$$
as $n \to \infty$. A classical way of studying this convergence is to take the Fourier transforms $\widehat{F}_n: \mathbb{R} \to \mathbb{R}$ of each $F_n$ and to apply a result from probability theory that gives a uniform bound on the difference $F_n-N$, where $N$ is our desired normal distribution, in terms of the $\widehat{F}_n$. This is the approach we employ.\\ \indent  
These Fourier transforms are given by
\begin{equation} \label{eq:4.1}
 \widehat{F}_n(t)  = \frac{1}{\# W_n} \sum_{g \in W_n} e^{it\varphi(g)n^{-1/2}}.
\end{equation}

\begin{lemma}
We have that, for the $\epsilon$ given in Lemma $6.1$,
\begin{equation}\label{eq:1}
 \widehat{F}_n(t) =  \frac{\sum_{i=1}^m L^n_{C_i,itfn^{-1/2}}\chi(\dot{0})}{\sum_{i=1}^m L^n_{C_i,0}\chi(\dot{0})} +o(1),
\end{equation}
when $|t| < \epsilon \sqrt{n}$. The above error term is uniform in $|t|<\epsilon \sqrt{n}$.
\end{lemma}

\begin{proof}
Setting $s=itn^{-1/2}$ in Lemma $6.1$ allows us to rewrite expression $(8.1)$ as 
$$\widehat{F}_n(t) = \frac{1}{\# W_n} \sum_{i=1}^m L^n_{C_i,itfn^{-1/2}}\chi(\dot{0}) +o(1).$$
Similarly, by setting $s=0$ in Lemma $6.1$, we have that
$$\#W_n \sim \sum_{i=1}^m L^n_{C_i,0}\chi(\dot{0}).$$
Combining these two identities proves the lemma.
\end{proof}

To obtain our central limit theorem with Berry-Esseen error term we want to make use of an inequality similar to the well-known `Basic Inequality' (see \cite{wiley} and the appendix). As mentioned above, this inequality allows us to study the convergence rate of our central limit theorem via the Fourier transforms of our distributions. The standard `Basic Inequality' applies to distributions with zero mean and for our purposes, we need a version of the inequality that applies to a sequence of distributions with varying means. We therefore amend the Basic Inequality to the following form. A proof is provided in the appendix.

\begin{proposition}
Let $H_n$ for $n\in \mathbb{Z}_{\ge0}$ be a sequence of distributions with Fourier transforms $\widehat{H}_n$ and means $E_n$. Write $N$ for the normal distribution with mean zero and variance $\sigma^2>0$ and suppose that $H_n - N$ vanishes at $\pm \infty$ for each $n \in \mathbb{Z}_{\ge0}$. Suppose there exists a sequence of positive real numbers $T_n >0$ and a constant $C>0$ such that 
$$\int_{-T_n}^{T_n} |\widehat{H}_n(t)| \ dt \le C,$$
for all $n\in \mathbb{Z}_{\ge0}.$ Then, there exists $K\ge0$ such that
\begin{equation}
\|H_n - N\|_{\infty} \le K \left( \int_{-T_n}^{T_n} \frac{1}{|t|} |\widehat{H}_n(t) - e^{-\sigma^2t^2/2}| \ dt + \frac{1}{T_n} + |E_n|e^{|E_n T_n|} \right),
\end{equation}
for all $n\in \mathbb{Z}_{\ge 0}$.
\end{proposition} 

We could apply this result directly to our distributions $F_n$, however, the error term in expression $(8.2)$ would lead to complications when comparing $\widehat{F}_n$ to $e^{-\sigma^2t^2/2}$ in the right hand side of $(8.3)$. Ideally, if we are to apply Proposition $8.2$ to a sequence of distributions $H_n$, we would like an exact expression for each $\widehat{H}_n$ in terms of transfer operators. To achieve this, we consider, instead of $F_n$, the following sequence of distributions,
\begin{align*}
H_n(x) =  \frac{1}{\# W_n + (m-1)\#N_n} \bigg( \# &\left\{ g \in W_n : \frac{\varphi(g)}{\sqrt{n}} \le x \right\}\\
&\hspace{0.7cm} + (m-1) \# \left\{ g \in N_n : \frac{\varphi(g)}{\sqrt{n}} \le x \right\} \bigg), 
\end{align*}
where $$\hspace{-4cm} N_n = \{g \in W_n : \text{ the path in $\mathcal{G}$ corresponding}$$
 $$\hspace{5.5cm} \text{to $g$ does not enter a maximal component}\}$$
and $\mathcal{G}$ has $m$ maximal components. \\ \indent
Since $\# N_n = O\left(e^{n(h-\delta)}\right)$ for some $\delta >0$, $\|F_n - H_n\|_{\infty}$ converges to zero exponentially quickly. Hence, to prove Theorem \ref{CLT}, it suffices to show the following.
\begin{proposition} We have that
$$H_n(x) = \frac{1}{\sqrt{2\pi} \sigma_\varphi} \int_{-\infty}^x e^{-t^2/2\sigma_\varphi^2} \ dt +  O\left(\frac{1}{\sqrt{n}}\right),$$
where the implied constant is independent of $x \in \mathbb{R}$ and $n \in \mathbb{Z}_{\ge 0}$.
\end{proposition}
We consider the distributions $H_n$, because each $\widehat{H}_n$ has an exact expression in terms of transfer operators.

\begin{lemma}\label{overcount}
For all $t\in \mathbb{R}$ and $n\in \mathbb{Z}_{\ge 0}$,
$$  \widehat{H}_n(t) =  \frac{\sum_{i=1}^m L^n_{C_i,itfn^{-1/2}}\chi(\dot{0})}{\sum_{i=1}^m L^n_{C_i,0}\chi(\dot{0})}.$$
\end{lemma}

\begin{proof}
We note that, for all $t \in \mathbb{R}$ and $n\in \mathbb{Z}_{\ge 0}$, 
$$\sum_{g \in W_n} e^{it\varphi(g)n^{-1/2}} + (m-1)\sum_{g \in N_n} e^{it\varphi(g)n^{-1/2}} = \sum_{i=1}^m L^n_{C_i,itfn^{-1/2}}\chi(\dot{0}).$$
Using this expression and the same proof as Lemma $8.1$ gives the required result. 
\end{proof}
We want to apply Proposition $8.2$ to the sequence $H_n$ and a suitable sequence $T_n$. Our aim is to show that for any sufficiently small $\epsilon>0$, Proposition $8.2$ holds for the pair $H_n$ and $T_n = \epsilon \sqrt{n}$.

\begin{lemma}
For any fixed sufficiently small $\epsilon >0$, there exists a constant $C >0$ depending only on $\epsilon$ such that
$$\int_{-\epsilon \sqrt{n}}^{\epsilon \sqrt{n}} |\widehat{H}_n(t)| \ dt \le C,$$
for all $n \in\mathbb{Z}_{\ge 0}.$
\end{lemma}

\begin{proof}
Since $\sum_{i=1}^m L^n_{C_i,0}\chi(\dot{0}) = \Theta(e^{nh})$,
$$\int_{-\epsilon \sqrt{n}}^{\epsilon \sqrt{n}} |\widehat{H}_n(t)| \ dt = O\left(e^{-nh} \int_{-\epsilon \sqrt{n}}^{\epsilon \sqrt{n}} \sum_{i=1}^m |L_{C_i,itfn^{-1/2}}^n \chi(\dot{0})| \ dt\right).$$
Hence it suffices to show that for each $i=1,...,m$, if $\epsilon>0$ is sufficiently small, then
\begin{equation}
e^{-nh} \int_{-\epsilon \sqrt{n}}^{\epsilon \sqrt{n}}|L_{C_i,itfn^{-1/2}}^n \chi(\dot{0})| \ dt = O(1),
\end{equation}
where the implied constant is independent of $n \in \mathbb{Z}_{\ge 0}$.\\
\indent Using the projections $Q_{i,k}$ and $Q_i$ we can write for sufficiently small $\epsilon$,
\begin{align}
L_{C_i,itfn^{-1/2}}^n \chi(\dot{0}) = &\sum_{k=0}^{p_i-1} e^{nP_i(itfn^{-1/2})} e^{2\pi i kn/p_i} Q_{i,k}(itn^{-1/2})\chi(\dot{0}) \nonumber \\ 
&\hspace{1.95cm}+ L_{C_i,itfn^{-1/2}}^n(I-Q_i(itn^{-1/2}))\chi(\dot{0}). 
\end{align}
Substituting this expression into the left hand side of $(8.4)$ implies that to prove $(8.4)$ it suffices to show that
\begin{equation}
e^{-nh} \int_{-\epsilon \sqrt{n}}^{\epsilon \sqrt{n}} \left| e^{nP_i(itfn^{-1/2})}Q_{i,k}(itn^{-1/2})\chi(\dot{0}) \right| dt = O(1)
\end{equation}
and
\begin{equation}
e^{-nh} \int_{-\epsilon \sqrt{n}}^{\epsilon \sqrt{n}} \left| L_{C_i,itfn^{-1/2}}^n(I-Q_i(itn^{-1/2})) \chi(\dot{0}) \right| dt = O(1)
\end{equation}
for each $i=1,...,m$, $k=0,...,p_i-1$ and that these error terms are independent of $n$.\\ \indent
To prove $(8.6)$, note that the Taylor expansion for the pressure $(4.3)$ implies that if $\epsilon$ is sufficiently small, then for all $|t| < \epsilon \sqrt{n}$,
$$|e^{nP_i(itfn^{-1/2}) - nh}| \le e^{-\sigma_\varphi^2t^2/4}.$$
Hence for fixed, sufficiently small $\epsilon$,
$$e^{-nh} \int_{-\epsilon \sqrt{n}}^{\epsilon \sqrt{n}} \left| e^{nP_i(itfn^{-1/2})}Q(itn^{-1/2})\chi(\dot{0}) \right| dt = O\left(\int_{-\epsilon\sqrt{n}}^{\epsilon \sqrt{n}} e^{-\sigma^2t^2/4} dt\right)= O(1).$$

To prove $(8.7)$, recall that, by Proposition $4.6$, if $\epsilon$ is sufficiently small, then for fixed $s$ with $|s|<\epsilon,$ there exists $\delta' >0$ such that
$$L_{C_i,sf}^n(I-Q_i(s)) \chi(\dot{0}) = O\left(e^{n(h-\delta')}\right),$$
where the implied constant is independent of $n \in \mathbb{Z}_{\ge 0}$. Since the maps $s \mapsto L_s$ and $s \mapsto Q_i(s)$ for $i=1,...,m$ are continuous (in fact analytic), at the cost of reducing $\epsilon$, we can find $\delta>0$ and $K>0$ such that
$$L_{C_i,sf}^n(I-Q_i(s)) \chi(\dot{0}) \le Ke^{n(h-\delta )},$$
for all $|s|<\epsilon$ and $n \in \mathbb{Z}_{\ge 0}$.
Hence
$$L_{C_i,itfn^{-1/2}}^n(I-Q_i(itfn^{-1/2})) \chi(\dot{0}) = O\left(e^{n(h-\delta )}\right),$$
where the implied constant is independent of $t$ and $n$ with $|t|<\epsilon \sqrt{n}$. Substituting this expression into the left hand side of $(8.6)$ gives the required decay rate. This concludes the proof.
\end{proof}

We have shown that Proposition $8.2$ applies to the pair $H_n$ and $T_n = \epsilon \sqrt{n}$ as long as $\epsilon >0$ is sufficiently small. The bound $(8.3)$ then provides us with a way of computing the decay rate of $\|H_n - N\|_\infty$, where $N$ is the normal distribution with mean $0$ and variance $\sigma_\varphi^2 >0$. We now turn our attention to the terms in $(8.3)$. We begin by studying the means $E_n$ of the distributions $H_n$. These means are given by
$$\sqrt{n} E_n=\int \varphi(g) \ d\tilde{\mu}_n$$
where
$$\tilde{\mu}_n = \frac{1}{\#W_n + (m-1)\#N_n}  \left(\sum_{g \in W_n} \delta_g + (m-1)\sum_{g \in N_n} \delta_g\right).$$
It follows easily from Theorem \ref{averaging} that $E_n \to 0$ as $n\to \infty$, further, we can quantify the rate of this convergence.

\begin{proposition}
We have that
$$E_n = O\left(\frac{1}{\sqrt{n}}\right).$$
\end{proposition}

\begin{proof}
This is a simple application of Theorem \ref{averaging}.
\end{proof}

We now study the decay rate of the first term in the right hand side of $(8.3)$. Our aim is to prove the following.

\begin{proposition}
For any fixed $\epsilon >0$ sufficiently small,
$$\int_{-\epsilon\sqrt{n}}^{\epsilon \sqrt{n}} \frac{1}{|t|} |\widehat{H}_n(t)-e^{-\sigma_\varphi^2t^2/2}| \ dt = O\left(\frac{1}{\sqrt{n}}\right),$$
where the implied constant is independent of $n \in \mathbb{Z}_{\ge 0}$.
\end{proposition}

We will break the proof of this proposition into two lemmas. We begin by studying the following difference 
$$\widehat{H}_n(t) - e^{-\sigma_\varphi^2t^2/2} = \frac{\sum_{i=1}^m \left( L_{C_i,itn^{-1/2}f}^n \chi(\dot{0}) - e^{-\sigma_\varphi^2t^2/2} L_{C_i,0}^n \chi(\dot{0})\right)}{\sum_{i=1}^m L_{C_i,0}^n \chi(\dot{0})}.$$
By Proposition $2.7$ we can write
$$\left|\widehat{H}_n(t) - e^{-\sigma_\varphi^2t^2/2} \right| \le C e^{-nh} \sum_{i=1}^m \left| L_{C_i,itn^{-1/2}f}^n \chi(\dot{0}) - e^{-\sigma_\varphi^2t^2/2} L_{C_i,0}^n \chi(\dot{0}) \right|,$$
where $C>0$ is a constant independent of $n \in \mathbb{Z}_{\ge 0}$. Hence to prove Proposition $8.7$ it suffices to show that for each $i=1,...,m$, if $\epsilon >0$ is sufficiently small,
\begin{equation}
\int_{-\epsilon\sqrt{n}}^{\epsilon \sqrt{n}} \frac{1}{|t|} \left| L_{C_i,itn^{-1/2}f}^n \chi(\dot{0}) - e^{-\sigma_\varphi^2t^2/2} L_{C_i,0}^n \chi(\dot{0}) \right| \ dt = O\left(\frac{1}{\sqrt{n}}\right).
\end{equation}
Substituting $(8.5)$ into $(8.8)$ we obtain (assuming that $\epsilon$ is sufficiently small),
$$e^{-nh} \int_{-\epsilon \sqrt{n}}^{\epsilon \sqrt{n}} \frac{1}{|t|} \left| L_{C_i,itn^{-1/2}f} \chi(\dot{0}) - e^{-\sigma_\varphi^2t^2/2}L_{C_i,0}^n\chi(\dot{0})\right| dt \le \text{I}(\epsilon)_n^i + \text{II}(\epsilon)_n^i,$$
where $\text{I}(\epsilon)_n^i $, $\text{II} (\epsilon)_n^i$ are given by

$$\sum_{k=0}^{p_i -1} e^{-nh}\int_{-\epsilon \sqrt{n}}^{\epsilon \sqrt{n}} \frac{1}{|t|} \left| e^{nP_i(itfn^{-1/2})}Q_{i,k}(itn^{-1/2})\chi(\dot{0}) - e^{-\sigma_\varphi^2t^2/2 + nh} Q_{i,k}(0)\chi(\dot{0})\right| dt,$$
$$e^{-nh}\hspace{-1mm}\int_{-\epsilon \sqrt{n}}^{\epsilon \sqrt{n}} \frac{1}{|t|} \left| L_{C_i,itn^{-1/2}f}^n(I\hspace{-2pt}-\hspace{-2pt}Q_i(itn^{-1/2}))\chi(\dot{0}) - e^{-\sigma_\varphi^2t^2/2 }L_{C_i,0}^n (I\hspace{-2pt}-\hspace{-2pt}Q_i(0))\chi(\dot{0})\right| dt$$
respectively. We have therefore shown that to prove Proposition $8.7$, it suffices to show that $\text{I}(\epsilon)_n^i$ and $\text{II}(\epsilon)_n^i$ decay at a $n^{-1/2}$ rate. The next two lemmas prove this.

\begin{lemma}
For any fixed sufficiently small $\epsilon>0$, 
$$\textnormal{I}(\epsilon)_n^i = O\left(\frac{1}{\sqrt{n}}\right).$$
\end{lemma}

\begin{proof}
It suffices to show that for any fixed sufficiently small $\epsilon >0$ and for all $i,k$, the quantity
$$e^{-nh} \int_{-\epsilon \sqrt{n}}^{\epsilon \sqrt{n}} \frac{1}{|t|}\left| e^{nP_i(itfn^{-1/2})} Q_{i,k}(itn^{-1/2})\chi(\dot{0}) - e^{-\sigma_\varphi^2t^2/2 +nh} Q_{i,k}(0)\chi(\dot{0})\right| dt $$
is $O\left(n^{-1/2}\right)$.
By the triangle inequality, this is a simple consequence of the following two estimates.\\ \indent
For any fixed sufficiently small $\epsilon>0$ ,
\begin{equation}
e^{-nh}\int_{-\epsilon \sqrt{n}}^{\epsilon \sqrt{n}}\frac{1}{|t|} \left|e^{nP_i(itfn^{-1/2})}Q_{i,k}(itn^{-1/2})\chi(\dot{0}) - e^{nP_i(itfn^{-1/2})}Q_{i,k}(0)\chi(\dot{0})\right| dt
\end{equation}
and
\begin{equation}
e^{-nh} \int_{-\epsilon \sqrt{n}}^{\epsilon \sqrt{n}} \frac{1}{|t|} \left| e^{nP_i(itfn^{-1/2})} Q_{i,k}(0)\chi(\dot{0}) - e^{-\sigma_\varphi^2t^2/2 +nh} Q_{i,k}(0)\chi(\dot{0}) \right| dt
\end{equation}
are both $O\left(n^{-1/2}\right)$.
To prove that $(8.9)$ decays at an $O(n^{-1/2})$ rate, recall that for each $i,k$ there exists bounded linear operators $\widetilde{Q}_{i,k}$ such that
$$Q_{i,k}(t) = Q_{i,k}(0) + t \widetilde{Q}_{i,k}(t)$$
for all $t$ sufficiently small. Also, from the Taylor expansion for the pressure $(4.3)$ (recall that we are assuming $\Lambda_\varphi =0$), we can assume that $\epsilon$ is sufficiently small so that for $|t| < \epsilon \sqrt{n}$,
$$|e^{nP_i(itfn^{-1/2}) - nh}| \le e^{-\sigma_\varphi^2t^2/4}.$$
Hence for fixed sufficiently small $\epsilon>0$, there exists $C>0$ such that
\begin{align*}
&e^{-nh}\left|e^{nP_i(itfn^{-1/2})}Q_{i,k}(itn^{-1/2})\chi(\dot{0}) - e^{nP_i(itfn^{-1/2})}Q_{i,k}(0)\chi(\dot{0})\right| \\
&\hspace{5cm}= \frac{|t|}{\sqrt{n}} \left|\widetilde{Q}_{i,k}\left(\frac{|t|}{\sqrt{n}}\right)\right|\left|e^{nP_i(tfn^{-1/2}) - nh}\right|\\ 
&\hspace{5cm}\le \frac{C|t|}{\sqrt{n}}e^{-\sigma_\varphi^2t^2/4},
\end{align*}
for all $|t| < \epsilon \sqrt{n}$.
Substituting this inequality into $(8.9)$ gives the result.\\ \indent
The required decay rate for $(8.10)$ can be proved analogously to Theorem $1$ in \cite{cp}. The proof is almost identical and hence we refer the reader to \cite{cp} for the proof.\\ \indent 
Combining $(8.9)$ and $(8.10)$ concludes the proof of the lemma.
\end{proof}

\begin{lemma}
For fixed small $\epsilon>0$,
$$\textnormal{II}(\epsilon)_n^i = O\left(\frac{1}{\sqrt{n}}\right).$$
\end{lemma}

\begin{proof}
Recall that by Lemma $4.4$, $L_{C_i,0}^n (I-Q_i(0))\chi(\dot{0}) = O\left(e^{n(h-\delta)}\right)$ for some $\delta >0$. Using this fact and the inequality
$\left|e^z - 1\right| \le |z| e^{|z|}$ it is easy to see that for any fixed sufficiently small $\epsilon$,
$$e^{-nh}\int_{-\epsilon \sqrt{n}}^{\epsilon \sqrt{n}} \frac{1}{|t|} \left| L_{C_i,0}^n(I-Q_i(0))\chi(\dot{0}) - e^{-\sigma_\varphi^2t^2/2} L_{C_i,0}^n(I-Q_i(0))\chi(\dot{0})\right| dt$$
is $O\left(n^{-1/2}\right)$.\\ \indent
Hence to conclude the proof of this lemma it suffices to show that for fixed small $\epsilon>0$ and for all $i$,
\begin{equation}
e^{-nh}\int_{-\epsilon \sqrt{n}}^{\epsilon \sqrt{n}} \frac{1}{|t|} \left| L_{C_i,itn^{-1/2}f}^n(I-Q(itn^{-1/2}))\chi(\dot{0}) -  L_{C_i,0}^n(I-Q_{i,k}(0))\chi(\dot{0})\right| dt 
\end{equation}
is $O\left(n^{-1/2}\right)$.\\ \indent
To obtain the required decay rate for $(8.11)$, we begin by defining operators $T_{i,n}(t)$ by
\begin{equation}
M^n L_{C_i,tf}^n(I-Q_i(t)) = M^nL_{C_i,0}^n(I-Q_i(0)) + T_{i,n}(t),
\end{equation}
where $M$ is the multiplication operator $Mg=e^{-h}g$.
To simplify notation in the following, let $L_t$ denote the operator $ML_{C_i,t}(I-Q_i(t))$. Note that the spectral radius of $L_0$ is strictly less than $1$. As discussed earlier, we can find (at the cost of reducing $\epsilon$), $0<\rho <1$ and $K>0$ such that 
$$\|L_s^n\| \le K\rho^n$$
for all $|s|<\epsilon$ and $n \in \mathbb{Z}_{\ge 0}$.\\ \indent
An operator version of the Mean Value Theorem (see Theorem $3.2$ of \cite{BP}) states that,
$$\|L_t^n - L_0^n \| \le |t| \sup_{0<l<1} \|D(L_{tl}^n)\|,$$
where $D(L_t)$ denotes the derivative of an operator $s \mapsto L_s$ at $t$. Furthermore, applying the Leibniz rule yields
$$D(L_{t}^n) = \sum_{k=1}^n L_t^{n-k} \  DL_t \ L_t^{k-1}.$$
Hence, for fixed, small $\epsilon$, 
\begin{align*}
\|T_{i,n}(itn^{-1/2})\| = \| L_{itn^{-1/2}}^n - L_0^n \| &\le |t|n^{-1/2} \sup_{0<l<1} \|D(L_{itn^{-1/2}l}^n)\|\\
&\le |t| n^{-1/2} C n \rho^{n}\\
&= C |t| \sqrt{n} \rho^{n},
\end{align*}
for some constant $C>0$ independent of $|t| < \epsilon \sqrt{n}$.

Now note that
$$e^{-nh}\left| L_{C_i,itn^{-1/2}f}^n(I-Q(itn^{-1/2}))\chi(\dot{0}) -  L_{C_i,0}^n(I-Q_{i,k}(0))\chi(\dot{0})\right| $$
 can be rewritten as
 $$ \left| T_{i,n}(itn^{-1/2}) \chi(\dot{0}) \right|.$$
We see that for fixed, sufficiently small $\epsilon>0$, there exists a constant $C>0$ (independent of $i$, $n$ and $t$) such that $(8.11)$ is bounded above by
 $$ C \int_{-\epsilon \sqrt{n}}^{\epsilon \sqrt{n}}\frac{1}{|t|} \sqrt{n} |t| \rho^n dt = 2C\epsilon n \rho^n.$$
 This clearly satisfies the required decay rate for $(8.11)$ and thus concludes the proof of the lemma.
\end{proof}

From these two lemmas, we deduce Proposition $8.7$. We are now ready to prove our central limit theorem.

\begin{proof} [Proof of Theorem \ref{CLT}]
By Lemma $8.5$, Proposition $8.6$ and Proposition $8.7$, there exists $\epsilon>0$ such that for $T_n = \epsilon \sqrt{n},$ the following hold.
\begin{enumerate}
\item The pair $H_n, T_n$ satisfy the conditions required to apply Proposition $8.2,$ with $N$ as the normal distribution with mean $0$ and variance $\sigma_\varphi^2 >0$.
\item $$\int_{-T_n}^{T_n} \frac{1}{|t|} |\widehat{H}_n(t)-e^{-\sigma_\varphi^2t^2/2}| \ dt = O\left(\frac{1}{\sqrt{n}}\right).$$
\item $$|E_n|e^{|T_n E_n|}= O\left(\frac{1}{\sqrt{n}}\right).$$
\end{enumerate}
Furthermore, the above implied error term constants are independent of $n \in \mathbb{Z}_{\ge 0}$. Proposition $8.2$ then implies that
$$\|H_n - N\|_{\infty} = O\left(\frac{1}{\sqrt{n}}\right),$$
proving Proposition $8.3$. As discussed in the paragraph preceding Proposition $8.3,$ this convergence implies that
$$\|F_n - N \|_\infty = O\left(\frac{1}{\sqrt{n}}\right)$$
as required.
\end{proof}

\section{Large Deviation Theorem}
In this section we prove our large deviation theorem. We begin by defining the following sequence of measures on $\Sigma_A$,
\begin{align*}
&\mu_n = \frac{1}{\#M_n} \sum_{z \in M_n} \delta_z,
\end{align*}

\noindent where $\delta_{x}$ denotes the Dirac measure based at $x$ and
$$M_n = \{z \in \Sigma_A: \sigma^n(z)=\dot{0}, z_0=\ast \text{ and }  z_{n-1}\neq 0\}.$$
We want to rephrase our large deviation result in terms of $f$ and $\mu_n$ on $\Sigma_A$. A simple calculation gives that
 $$\frac{1}{\#W_n}\#\left\{g \in W_n : \left|\frac{\varphi(g)}{n} \right| > \epsilon \right\} = \mu_n\left\{ z \in \Sigma_A: \left|\frac{f^n(z)}{n}\right| > \epsilon\right\}.$$

\noindent Hence to prove Theorem $1.3$, it suffices to show that for each $\epsilon >0$,
$$\limsup_{n \to \infty} \frac{1}{n} \log \mu_n\left\{ z \in \Sigma_A: \left|\frac{f^n(z)}{n}\right| > \epsilon\right\} <0.$$

We need the following lemma.

\begin{lemma}
Fix $\epsilon >0$. Then, there exists $\rho >0$ and $k \in \{1,...,m\}$ such that for fixed $t\in \mathbb{R}$ satisfying $0< t < \rho$, 
$$ \int e^{tf^n(z)} d \mu_n = O\left(e^{-nh +nt\epsilon/2 + nP_k(tf)}\right).$$
The implied constant depends on $t$ and $\epsilon$ but not on $n$.
\end{lemma}

\begin{proof}
Let $\delta, \epsilon$ be as in Lemma $6.1$. Take $0<\rho < \epsilon$. For $0<t<\rho$ and for all $i$, $\left|P_i(tf)-h\right| < \delta$. By Lemma $6.1$ we can write
$$\# W_n \int e^{tf^n(z)} d \mu_n = \sum_{i=1}^m L^n_{C_i,tf}\chi(\dot{0}) +O(e^{n(h-\delta)}).$$
By the Spectral Radius Theorem (Gelfand's formula) we can take $C_t>0$, depending on $t$ but not $i$, such that
$$\| L_{C_i,tf}^n \| \le C_t e^{n(P_i(tf) + \epsilon t/2)}$$
for all $n \in \mathbb{Z}_{\ge 0}$ and $i =1,...,m$. Combining these observations gives that
\begin{align*}
\int e^{tf^n(z)} d \mu_n &= \frac{\sum_{i=1}^m L^n_{C_i,tf}\chi(\dot{0})}{\#W_n} +O(e^{-n\delta})\\
&=O\left(e^{-nh+nt\epsilon/2} \sum_{i=1}^m e^{nP_i(tf)}, e^{-n\delta}\right)\\
&=O\left(e^{-nh +nt\epsilon/2} \sum_{i=1}^m e^{nP_i(tf)}\right).
\end{align*}
We now recall that, by Proposition $4.5$, the maps $ t \mapsto e^{P_i(tf)}$ for $i=1,...,m$, are real analytic. Hence there exists $\xi >0$ and $k \in \{1,...,m\}$ such that for all $0 <  t < \xi$,
$$\max_{i=1,...,m} \left\{ e^{P_i(tf)}\right\} = e^{P_k(tf)}.$$
By reducing $\rho$, if necessary, so that it is less that $\xi$, we see that for fixed $0<t<\rho$,
\begin{align*}
\int e^{tf^n(z)} d \mu_n &=O\left(e^{-nh +nt\epsilon/2} \sum_{i=1}^m e^{nP_i(tf)}\right)\\
&=O\left(e^{-nh +nt\epsilon/2 + nP_k(tf)}\right),
\end{align*}
as required.
\end{proof}

The same proof as the previous lemma gives the following.
\begin{lemma}
Fix $\epsilon>0$. Then, there exists $\rho' < 0$ and $k' \in \{1,...,m\}$ such that for fixed $t\in \mathbb{R}$  satisfying $\rho' < t <0$, 
$$ \int e^{tf^n(z)} d \mu_n = O\left(e^{-nh - nt\epsilon/2 + nP_{k'}(tf)}\right).$$
The implied constant depends on $t$ and $\epsilon$ but not on $n$.
\end{lemma}

We are now ready to prove our large deviation theorem.

\begin{proof} [Proof of Theorem $1.3$]

\noindent Fix $\epsilon >0$. Let $\rho$ and $k$ be those chosen in Lemma $9.1$. Define $b(s) = -s\epsilon/2 - h + P_k(sf)$. Note that $b(0)=0$ and
$$b'(0) = -\epsilon/2 + \frac{d}{ds} P_k(sf)\big|_{s=0} = -\epsilon/2 + \Lambda_\varphi=  - \epsilon/2 <0.$$
Hence we can choose $0 < t < \rho$ such that $b(t) <0$. Fix $t$ at this value, then,
\begin{align*}
\mu_n\left\{ z \in \Sigma_A: \frac{f^n(z)}{n} > \epsilon\right\} &\le \int e^{t(f^n(z) - n\epsilon)} d\mu_n\\
&=e^{-tn\epsilon} \int e^{tf^n(z)} d \mu_n\\
&\le \tilde{C}_t e^{-tn\epsilon - nh + tn\epsilon/2 + nP_k(tf) }\\
&= \tilde{C}_t e^{nb(t)},
\end{align*}
where the second inequality in the above follows from Lemma $9.1$ and $\tilde{C}_t$ is the constant associated to the error term from this lemma.\\ \indent
Hence,
$$\limsup_{n \to \infty} \frac{1}{n} \log \mu_n\left\{ z \in \Sigma_A: \frac{f^n(z)}{n} > \epsilon\right\} \le  b(t) <0.$$
The inequality
$$\limsup_{n \to \infty} \frac{1}{n} \log \mu_n\left\{ z \in \Sigma_A: \frac{f^n(z)}{n} < -\epsilon\right\} <0$$
can be proven in a similar way, this time using Lemma $9.2$ instead of Lemma $9.1$. By our earlier discussion, this concludes the proof.
\end{proof}

\section{Statistics of the abelianisation homomorphism}

In this section we prove Theorem $1.6$. To do so, we generalise our current methods to the multidimensional setting. That is, we show that our methods apply to functions $\varphi: G \to \mathbb{R}^k$ that satisfy Condition $(1)$ and Condition $(2)$ component wise. We begin by recalling the multidimensional central limit theorem for subshifts of finite type. Let $\langle \cdot, \cdot \rangle$ denote the Euclidean inner product. \\ \indent
Suppose $\Sigma_M$ is an irreducible subshift of finite type and $f: \Sigma_M \to \mathbb{R}^k$ a function with components that belong to $F_\theta(\Sigma_M)$ for some $0 < \theta <1$. Then, there exists a covariance matrix $\Sigma \in M_k(\mathbb{R})$ and $\Lambda \in \mathbb{R}^k$ such that for any $A \subset \mathbb{R}^k$,
$$\mu \left\{ x \in \Sigma_M: \frac{f^n(x)-n\Lambda}{\sqrt{n}} \in A\right\} \to \frac{1}{(2\pi \ \text{det}(\Sigma))^{k/2}} \int_{A} e^{-\langle x , \Sigma x \rangle /2} \ dx $$
where $\mu$ is the measure of maximal entropy for $(\Sigma_M, \sigma)$. Furthermore, the following are equivalent.
\begin{enumerate}
\item The above central limit theorem is non-degenerate,
\item $\Sigma$ is positive definite, 
\item $\langle t, f\rangle$ is not cohomologous to a constant for any $t \in \mathbb{R}^k\backslash \{0\}$, 
\item for each $t\in \mathbb{R}^k\backslash\{0\}$ the set $\left\{\langle t , (f^n(x) - n\Lambda)\rangle : x \in \Sigma_M, n\in\mathbb{Z}_{\ge 0}\right\}$ is unbounded.
\end{enumerate}
Let $L_{M, \langle s,f\rangle}$ denote the transfer operator acting on $F_\theta(\Sigma_M)$ defined in Definition $2.2$, i.e.
$$L_{M, \langle s, f \rangle} w(x)=\sum_{\sigma y =x} e^{\langle s, f(y) \rangle } w(y)$$
for $w \in F_\theta(\Sigma_M)$. Proposition $4.5$ implies that for all sufficiently small $s \in \mathbb{C}^k$, the transfer operator $L_{\langle s, f\rangle}$ has $p$ simple maximal eigenvalues of the form $e^{2 \pi i j / p} e^{P(\langle s, f \rangle)}$ for $j=1,...,p$ where $p$ is the period of $M$ and $s \mapsto P(\langle s,f \rangle)$ is analytic (in the multi variable sense) in a neighbourhood of the origin. The constant $\Lambda$ and covariance matrix $\Sigma$ have entries
$$\Lambda_i = \frac{\partial}{\partial s_i} \bigg|_{s=0} P(\langle s,f \rangle) \hspace{1mm} \text{ and } \hspace{1mm} \Sigma_{i,j} = \frac{\partial^2}{\partial s_i \partial s_j} \bigg|_{s=0} P(\langle s, f \rangle )$$
for $i,j \in \{1,...,k\}$ and where $s = (s_1,...,s_k)$.\\ \indent
Using the same arguments as in Sections $5$, we can deduce similar statements concerning the spectra of the operators $L_{C_i, \langle s,f \rangle}$.
\begin{proposition}
There exists $\epsilon >0$ such that for all $\|s\| <\epsilon $ the operators $L_{C_i, \langle s,f \rangle}$ for $i=1,...,m$ each have $p_i$ simple maximal eigenvalues $e^{2 \pi i j/p_i} e^{P_i(\langle s,f \rangle)}$ for $j=0,...,p_i-1$, where each $s \mapsto e^{P_i(\langle s,f \rangle)}$ is analytic in $\|s\| <\epsilon$.
\end{proposition}
Futhermore, the argument of Calegari and Fujiwara presented in Proposition $4.8$ can be applied to compare the pressure functions $P_i(\langle s,f \rangle)$ for $i=1,...,m$. The following result can be obtained using the same argument used to prove Proposition $4.8$. The required modification to the proof is simple, we need only replace the use of the central limit theorem for subshifts of finite type with the multidimensional version stated above.
\begin{proposition}
Given $\alpha, \beta \in \{1,...,k\}$ the quantities
$$(\Lambda_{\varphi})_\alpha : =\frac{\partial}{\partial s_\alpha} \bigg|_{s=0} P_i(\langle s,f \rangle) \hspace{1mm} \text{ and } \hspace{ 1mm} (\Sigma_\varphi)_{\alpha,\beta} := \frac{\partial^2}{\partial s_\alpha \partial s_\beta} \bigg|_{s=0} P_i(\langle s,f \rangle)$$
do not depend on the maximal component $B_i$. Furthermore, for each $i=1,...,m$ and $\|s\|<\epsilon$,
$$P_i(\langle s,f \rangle ) = h+ \Lambda_\varphi s + \langle s,  \Sigma_\varphi s \rangle + O(\|s\|^3)$$
as $s \to 0$.
\end{proposition}
We now turn our attention to the non-degeneracy criteria in the multidimensional setting. Lemma $5.1$ can be easily generalised using the multidimensional criteria for degeneracy stated above. We obtain the following result.
\begin{proposition}
Let $\Sigma_\varphi$ be the covariance matrix defined above. Then $\Sigma_\varphi$ is positive definite if and only if for each non-zero $t \in \mathbb{R}$, the function $\langle t, \varphi(\cdot) - \Lambda_\varphi | \cdot|\rangle : G \to \mathbb{R}$ is  unbounded.
\end{proposition}
We are now ready to prove a multidimensional central limit theorem.
\begin{theorem}\label{mdlt2}
Suppose $\varphi: G\to\mathbb{R}^k$ satisfies Condition $(1)$ and Condition $(2)$ componentwise. Then there exists $\Lambda_\varphi \in \mathbb{R}^k$ and a symmetric matrix $\Sigma_\varphi \in M_k(\mathbb{R})$ such that
$$\frac{1}{\#W_n} \# \left\{ g \in W_n : \frac{\varphi(g)-\Lambda_\varphi n}{\sqrt{n}} \in A \right\} \rightarrow \frac{1}{(2\pi \ \det(\Sigma_\varphi))^{k/2}} \int_{A} e^{-\langle x, \Sigma_\varphi x \rangle /2} \ dx $$
as $n\to\infty$. Furthermore, $\Sigma_\varphi$ is positive definite if and only if for each non-zero $t\in\mathbb{R}^k$ the function $\langle t, \varphi(\cdot) - \Lambda_\varphi | \cdot|\rangle : G \to \mathbb{R}$ is unbounded.
\end{theorem}

\begin{proof}
We have already discussed the non-degeneracy criteria. We therefore just need to prove the central limit theorem. As in the previous sections, we may assume that $\Lambda_\varphi =0$. It then suffices, by L\'evy's Continuity Theorem, to show that for each $t \in \mathbb{R}^k$,
$$\widehat{F}_n(t) \to e^{-\langle t,\Sigma_\varphi t \rangle /2}$$
as $n\to\infty$, where
$$\widehat{F}_n(t) = \frac{1}{\# W_n} \sum_{g\in W_n} e^{i \langle t, \varphi \rangle n^{-1/2}}.$$ 
Using a multidimensional analogue of Lemma $6.1$ (which can be proved in the same way as the one-dimensional version), we can write, for all $\|t\|n^{-1/2}$ sufficiently small,
$$ \sum_{g\in W_n} e^{i \langle t, \varphi \rangle n^{-1/2}} = \sum_{i=1}^m L_{C_i, i\langle t, f \rangle n^{-1/2}}^n \chi(\dot{0}) + o(e^{nh})$$
as $n \to \infty$.
Hence,
$$ \widehat{F}_n(t) = \frac{ \sum_{i=1}^m L_{C_i, i\langle t, f \rangle n^{-1/2}}^n \chi(\dot{0}) + o(e^{nh})}{ \sum_{i=1}^m L_{C_i,0}^n\chi(\dot{0}) + o(e^{nh})}$$
as $n\to\infty$. Using the projections $Q_{i,k}$ and $Q_i$ for $i=1,...,m, k=0,...,p_i-1$, we can write
$$\widehat{F}_n(t) = e^{-\langle t,\Sigma_\varphi t \rangle /2} \ G_n(t)$$
where
$$G_n(t) =  \frac{ \sum_{i=1}^m \sum_{k=0}^{p_i-1} e^{n P(\langle itn^{-1/2}, f \rangle ) + \langle t,\Sigma_\varphi t \rangle /2}e^{2 \pi i k n /p_i} Q_{i,k}(itn^{-1/2})\chi(\dot{0}) + o(1) }{ \sum_{i=1}^m \sum_{k=0}^{p_i-1} e^{2 \pi i k n/p_i} Q_{i,k}(\chi)(\dot{0}) + o(1)}.$$
By the analyticity of the $Q_{i,k}$, for each $i =1,...,m$ and $k=0,...,p_i-1$, $Q_{i,k}(t) = Q_{i,k}(0) + O(\|t\|)$. Also, using the Taylor expansions for the pressures from Proposition $10.2$, for each $i=1,...,m$, $n P(\langle itn^{-1/2}, f \rangle ) + \langle t,\Sigma_\varphi t \rangle /2 = O(n^{-1/2})$. Combining these facts gives that
$$ G_n(t)=\frac{ \sum_{i=1}^m \sum_{k=0}^{p_i-1} e^{n P(\langle itn^{-1/2}, f \rangle ) + \langle t,\Sigma_\varphi t \rangle /2}e^{2 \pi i k n /p_i} Q_{i,k}(0)\chi(\dot{0}) + o(1) }{ \sum_{i=1}^m \sum_{k=0}^{p_i-1} e^{2 \pi i k n/p_i} Q_{i,k}(0)\chi(\dot{0}) + o(1)}$$ and so for each $t\in\mathbb{R}$, $G_n(t) \to 1$ as $n\to\infty$. Hence $\widehat{F}_n(t) \to e^{-\langle t,\Sigma_\varphi t \rangle /2}$ as $n\to\infty$ as required.
\end{proof}
We can now deduce Theorem $1.6$ as a corollary of the above result. Suppose that the abelianisation of $G$ is isomorphic to $\mathbb{Z}^k \oplus \text{Torsion }$ for some $k \ge 1$. Fix an isomorphism taking the non-torsion part of $G/[G,G]$ to $\mathbb{Z}^k$ and let $\varphi: G \to \mathbb{Z}^k$ be the induced homomorphism.
\begin{proof} [Proof of Theorem $1.6$]
To conclude the proof of Theorem $1.6$ we need to show that $\Lambda_\varphi=0$ and $\Sigma_\varphi$ is positive definite. To see that $\Lambda_\varphi =0$ note that for each $j=1,\ldots,k$ the $j$th coordinate of $\Lambda_\varphi$ is the mean $\Lambda_{\varphi_j}$ of the homomorphism $\varphi_j$ obtained by projecting $\varphi$ to its $j$th coordinate. By Theorem \ref{averaging} and a simple symmetry argument $\Lambda_{\varphi_j} =0$ for all $j=1,\ldots,k$. This concludes the first part of the proof. For the second part we need to show that $\langle t, \varphi\rangle$ is  unbounded for any $t \in \mathbb{R}^k \backslash \{0\}$. Since $\varphi$ is surjective onto $\mathbb{Z}^k$, the function $\psi_t: G \to \mathbb{R}$ defined by $ \psi_t = \langle t, \varphi\rangle$ is a non-trivial group homomorphism for any $t \in\mathbb{R}^k\backslash\{0\}$ and the result follows.
\end{proof}

\begin{remark}
The above proof applies to any surjective group homomorphism $\varphi: G \to \mathbb{Z}^k$.
\end{remark}

\section{Local limit theorem}
In this section we prove our local limit theorem, Theorem \ref{ndllt}. Suppose $\varphi : G \to \mathbb{R}$ is a group homomorphism satisfying the hypothesis of Theorem \ref{ndllt}. As in the other sections, we want to  study the function $f: \Sigma_A \to \mathbb{R}$ corresponding to $\varphi$. We begin by recalling the following definition.
\begin{definition}
We say that $f \in F_\theta$ is lattice if there exists $a,b \in \mathbb{R}$ such that
$$\{ f^n(x) - an : x\in \Sigma_A, n \in \mathbb{Z}_{\ge 0} \text{ with } \sigma^n(x) =x \} \subseteq b \mathbb{Z}.$$
\end{definition}
We want to prove that if $f$ is related to $\varphi$ via Condition $(1)$, then the restriction of $f$ to each maximal component is non-lattice. This will allow us to deduce important spectral properties for the transfer operators $L_{C_j,itf}$ where $t$ is real. The aim of the next couple of lemmas is to prove this. Recall that three real numbers $x,y,z$ are said to be rationally independent if the only solution $(\alpha, \beta, \gamma) \in \mathbb{\mathbb{Q}}^3$ to $\alpha x + \beta y+ \gamma z =0$, is the trivial solution with $\alpha=\beta=\gamma =0$.
\begin{lemma}\label{triple}
Suppose that there exists $g_1,g_2, g_3$ in $G$ such that $\varphi(g_1),\varphi(g_2), \varphi(g_3)$ form a rationally independent triple. Then, for any $a,b \in \mathbb{R}$, $H_{a,b} = \varphi^{-1}(a\mathbb{Z} \oplus b\mathbb{Z})$ is an infinite index subgroup of $G$, i.e. $|H_{a,b} : G| =\infty$.
\end{lemma}

\begin{proof}
For each $a,b \in \mathbb{R}$ there is $g \in G$ with $\varphi(g) \notin a \mathbb{Q} \oplus b \mathbb{Q}$. Indeed, if no such $g$ exists then we can find $x_i,y_i \in \mathbb{Q}$ for $i=1,2,3$ such that
$$\varphi(g_i) = ax_i + by_i \ \ \ \text{ for} \ \ \ i=1,2,3.$$
Eliminating $a$ and $b$ would imply that the $\varphi(g_i)$ are rationally dependent contrary to our assumption.\\ \indent
Now consider for $k,l \in \mathbb{Z}$ the cosets $g^k H_{a,b}, g^l H_{a,b}$ for $g \notin a \mathbb{Q} \oplus b \mathbb{Q}$ . If these cosets coincide then $g^{k-l} \in H_{a,b}$, $(k-l)\varphi(g) \in a \mathbb{Z} \oplus b \mathbb{Z}$ and so $\varphi(g) \in a \mathbb{Q} \oplus b \mathbb{Q}$. This contradiction implies that $g^k H_{a,b}$ and $g^l H_{a,b}$ are distinct for $k \neq l$. Hence $|H_{a,b} : G| =\infty$ as required.
\end{proof}

We now require the following result of Gou\"ezel, Math\`eus and Maucourant.
\begin{proposition} [Theorem $4.3$ \cite{gouz}] \label{gouzthm}
Suppose $G$ is a non-elementary hyperbolic group equipped with a finite generating set and $H< G$ is an infinite index subgroup of $G$. Then the density of $H$ with respect to $W_n$ is zero, i.e.
$$\lim_{n\to \infty} \frac{\#(W_n \cap H)}{\#W_n} = 0.$$
\end{proposition}
Using this and the previous lemma we deduce the following.

\begin{lemma} \label{cont}
For each $a,b \in \mathbb{R}$ there exist $D \in \mathbb{Z}_{\ge 0}$ such that
$$ \#\{g \in W_{Dn} : \varphi(g) - aDn \in b\mathbb{Z}\} = o(\#W_{Dn})$$
as $n\to\infty$.
\end{lemma}

\begin{proof}
Notice that
$$\{g \in G: \varphi(g) -a|g| \in b\mathbb{Z}\} \subseteq H_{a,b}.$$
Hence,
$$\#\{g \in W_n: \varphi(g) -an \in b\mathbb{Z}\} \le \#(W_n \cap H_{a,b}).$$
If $\varphi$ satisfies the hypotheses of Lemma \ref{triple}, then we may then apply Proposition \ref{gouzthm} to conclude that
$$\#(W_n \cap H_{a,b}) = o(\#W_n)$$
and the result follows.  \\ \indent
Otherwise, the image of $\varphi$ is $c\mathbb{Z} \oplus d\mathbb{Z}$ for some rationally independent $c,d \in \mathbb{R}$. We can assume that $a\mathbb{Z} \cap (c\mathbb{Z} \oplus d\mathbb{Z})$ is non-empty, since if it is empty, $H_{a,b}$ has infinite index and we can apply the same argument used above. Fix $D \in \mathbb{Z}_{\ge 0}$ such that $aD \in a\mathbb{Z}\cap (c\mathbb{Z} \oplus d\mathbb{Z})$. 
Then note that
$$\{g \in W_{Dn} : \varphi(g) - aDn \in b\mathbb{Z}\} \subset \{g \in W_{Dn}: \widetilde{\varphi}\circ\varphi(g) = \widetilde{\varphi}(aDn)\}$$
where $\widetilde{\varphi} : c\mathbb{Z} \oplus d\mathbb{Z}\to c\mathbb{Z} \oplus d\mathbb{Z}/(b\mathbb{Z}\cap (c\mathbb{Z} \oplus d\mathbb{Z})) = K_b$ is the quotient homomorphism. We have that $K_b$ is necessarily isomorphic to $\mathbb{Z} \oplus \text{Torsion}$ or $\mathbb{Z}^2$ depending on whether $b\mathbb{Z}\cap (c\mathbb{Z} \oplus d\mathbb{Z})$ is trivial. Let $\varphi': G \to \mathbb{Z}$ be the composition  $\psi \circ \widetilde{\varphi}\circ \varphi$ where $\psi: K_b \to \mathbb{Z}$ is a homomorphism that projects $K_b$ to a $\mathbb{Z}$ factor. We then have that
$$\#\{g \in W_{Dn} : \varphi(g) -aDn \in b\mathbb{Z}\} \le \#\{g \in W_{Dn} : \varphi'(g) - n(\psi \circ \widetilde{\varphi}(aD)) =0\}.$$
and so we need to show that
$$\#\{g \in W_{Dn} : \varphi'(g) - n(\psi \circ \widetilde{\varphi}(aD)) =0\} = o(\#W_{Dn}).$$
 This follows from Corollary \ref{homomcoro} if $\psi \circ \widetilde{\varphi}(aD)  = 0$ and Theorem \ref{ldt}  if $ \psi \circ \widetilde{\varphi}(aD) \neq 0$. This concludes the proof.
\end{proof}

We can now deduce the required properties of $f$.

\begin{lemma}
For each maximal component $B_j$ the restriction of $f$ to $\Sigma_{B_j}$, $f_j,$ is non-lattice.
\end{lemma}

\begin{proof}
Suppose $f_j$ is lattice. We can then find $a,b \in \mathbb{R}$ such that
$$\{f_j^n(x) -na : \sigma^n( x)=x, x \in \Sigma_{B_j}\} \subseteq b \mathbb{Z}.$$
Since $\#\{x \in \Sigma_{B_j}: \sigma^{np_j}(x)=x \}$ grows like $\lambda^{np_j}$, the correspondence between $G$ and $\Sigma_A$ implies that $\#\{g \in W_{np_j} : \varphi(g)-np_ja \in b\mathbb{Z}\} \ge C \lambda^{np_j}$ for some $C>0$. We then have that, for any integer $D$,
$$\limsup_{n\to\infty} \frac{1}{\#W_{Dn}} \#\{g \in W_{Dn} : \varphi(g)-aDn \in b\mathbb{Z}\} >0.$$
This contradicts Lemma $\ref{cont}$ and so the result follows.
\end{proof}

Using this lemma, we deduce the following.
\begin{proposition}\label{latticespec}
Suppose $\varphi:G\to \mathbb{R}$ is a group homomorphism with dense image. Then for all $t \in \mathbb{R}\backslash \{0\}$ and each $j=1,\ldots,m$ , the spectral radius of $L_{C_j,itf}$ is strictly less than $e^h$.
\end{proposition}

\begin{proof}
When $C_j$ consists of a single connected component, it is well known that the spectral radius of $L_{C_j,itf}$ is less than  or equal to $e^h$ for all $t \in \mathbb{R}$. The non-lattice condition guarantees that for all $t \in \mathbb{R}\backslash \{0\}$ this inequality is strict \cite{pp}. When $C_j$ is not a single component, $C_j$ contains a component with spectral radius $e^h$ and all other components have spectral radius strictly less than $e^h$. We can then, by Lemma $2$ of \cite{oc}, apply the above result component-wise to deduce our result.  
\end{proof}

We are now ready to prove Theorem \ref{ndllt}. Since our method follows that of \cite{RE}, we will sketch the proof and highlight where our work is needed.

\begin{proof}[Proof of Theorem \ref{ndllt}]
We sketch a proof. Recall that $\sigma^2_\varphi >0$. Theorem \ref{ndllt} is concerned with the asymptotics of
$$ \frac{1}{\#W_n} \sum_{g\in W_n} \chi_{[a,b]}(\varphi(g)),$$
where $\chi_{[a,b]}$ is the indicator function on $[a,b]$ for $a,b \in \mathbb{R}$.  We first consider this expression when $\chi_{[a,b]}$ is replaced by an integrable function $\phi_{[a,b]}: \mathbb{R} \to \mathbb{C}$ that has Fourier transform $\widehat{\phi}_{[a,b]}$ that is compactly supported and satisfies $\widehat{\phi}_{[a,b]}(t) = \widehat{\phi}_{[a,b]}(0) + O(|t|)$. Using Fourier inversion we can write
$$ \sum_{g\in W_n} \phi_{[a,b]}(\varphi(g)) = \frac{1}{2\pi } \int_{\mathbb{R}} \sum_{g \in W_n} e^{ it \varphi(g) } \widehat{\phi}_{[a,b]}(t) \ dt.$$
Then, using Lemma \ref{express} and the same over-counting argument used to prove Theorem \ref{CLT} (see Lemma \ref{overcount}), we can assume $$\sum_{g \in W_n} e^{it \varphi(g) }$$ has an exact expression in terms of the transfer operators. We can then write, 
$$\sum_{g\in W_n} \phi_{[a,b]}(\varphi(g)) = \frac{1}{2\pi } \int_{\mathbb{R}} \widehat{\phi}_{[a,b]}(t) \sum_{j=1}^m L_{C_j,itf}^n\chi(\dot{0}) \ dt.$$
Then, using Proposition \ref{latticespec} and Lemma \ref{express}, we show that there exists $\epsilon >0$ such that the lead terms describing the growth of this quantity are
$$\frac{1}{2 \pi} \int_{[-\epsilon,\epsilon]} \widehat{\phi}_{[a,b]}(t) e^{2\pi i kn/p_j} e^{nP_j(itf)}Q_{j,k}(it)\chi(\dot{0}) \ dt$$
for all pairs $j,k$. We can then apply the arguments presented in \cite{RE} to show that, for each $j,k$ this quantity grows asymptotically like
$$ \frac{\int \phi_{[a,b]}(t) \ dt \ \  e^{2 \pi i kn/p_j} Q_{j,k}(0)\chi(\dot{0})}{\sqrt{2\pi}\sigma_\varphi \sqrt{n}} \ e^{nh}$$
where we have used that $\sigma_\varphi$ is independent of the maximal component. We normalise by $\#W_n$ and write $\#W_n$ in terms of transfer operators (see the proof of Theorem \ref{mdlt2}) to see that
$$ \frac{1}{\#W_n} \sum_{g\in W_n} \phi_{[a,b]}(\varphi(g)) \sim \frac{\int \phi_{[a,b]}(t) \ dt}{\sqrt{2\pi}\sigma_\varphi \sqrt{n}}$$
as $n\to\infty$. Using a standard approximation argument we can remove the assumptions on $\phi_{[a,b]}$ and show that the above convergence holds when $\phi_{[a,b]}$ is replaced by any smooth positive function of compact support. Lastly we use one further standard approximation argument to deduce that the same converges holds when we replace $\phi_{[a,b]}$ with $\chi_{[a,b]}$. This concludes the proof.
\end{proof}

As mentioned in the introduction, the hypothesis of Theorem \ref{ndllt} is satisfied, in some sense, by almost every homomorphim $\varphi : G \to \mathbb{R}$. We will now explain what we mean by this. Note that, since every homomorphism $\varphi : G \to \mathbb{R}$ factors through the abelianisation $G/[G,G]$ of $G$, any homomorphism is of the form $g \mapsto \langle \varphi_{ab}(g) , v \rangle$ where $\varphi_{ab} : G \to \mathbb{Z}^k $ is the abelianisation homomorphism post-composed with the projection to the non-torsion factor of $G/[G,G]$, and $v$ is a vector in $\mathbb{R}^k$. We can therefore naturally identify the space of homomorphisms $\text{Hom}(G,\mathbb{R})$ with $\mathbb{R}^k$ where $k \in \mathbb{Z}$ is the rank of the abelianisation  of $G$. As long as $k \ge 2$ then we can find homomorphisms in $\text{Hom}(G,\mathbb{R})$ that satisfy our theorem, as these homomorphisms correspond to vectors $v \in \mathbb{R}^k$ that have two entries that form a rationally independent pair. Furthermore since rationally dependent pairs lie in a countable collection of planes of codimension at least $1$ in $\mathbb{R}^k$ for $k \ge 2$, the set of vectors in $\mathbb{R}^k$ that correspond to homomorphisms that satisfy our theorem have complement in $\mathbb{R}^k$ with Lebesgue measure zero. In this sense almost all homomorphisms satisfy the hypotheses of Theorem \ref{ndllt}.\\ \indent

We now prove a local limit theorem for the displacement function associated to certain `nice' actions. More specifically, for the rest of this section we are interested in actions on pinched Hadamard surfaces. A pinched Hadamard surface is a connected, simply connected, Riemannian surface with all sectional curvatures bounded above by $-1$. A group $G$ that acts on such a surface is said to by fuchsian if it acts properly discontinuously, freely and by isometries. See \cite{db} for precise definitions of these objects and for the others used throughout the rest of this section. Our aim now is to prove the following.
\begin{theorem} \label{thmhad}
Suppose that a fuchsian group $G$ (equipped with a finite symmetric generating set) acts convex cocompactly on a pinched Hadamard surface $X$ with origin $o \in X$. Then there exists $\sigma^2>0$ such that for $a,b \in \mathbb{R}$, $a<b$,
$$\frac{1}{\#W_n} \#\left\{ g \in W_n : d(o,go) - n\Lambda \in [a,b]\right\} \sim \frac{b-a}{\sqrt{2\pi} \sigma \sqrt{n}}$$
as $n \to\infty$.
\end{theorem}
Suppose for the rest of this section that $G$ and $X$ are as in the above theorem. It follows from the Svarc-Milnor Lemma that $G$ is hyperbolic and further, from our discussion in Section 2, that the displacement function satisfies Condition (1) and (2). We restrict our study to these actions because we have, in this setting, a good understanding of the length spectrum. Recall that the length spectrum for the action of $G$ on $X$ is the set of possible translation lengths, where, given $g \in G$, the translation length of $g$ is
$$\tau(g) = \lim_{n\to\infty} \frac{d(o,g^no)}{n}.$$
This limit exists by subadditivity.  Let $r: \Sigma_A \to \mathbb{R}$ be the function related to the displacement function via Condition (1). We would like to use arguments involving the length spectrum to deduce non-lattice properties for $r$. The following definition and subsequent lemma allow us to do this. 
\begin{definition}
Let $v$ be a vertex in $\mathcal{G}$. The loop semi-group $L_v$ associated to $v$ is the semi-group consisting of group elements $g \in G$ that correspond (under the labeling $\rho$ from Definition $2.5$) to a loop in $\mathcal{G}$ starting (and also ending) at $v$.
\end{definition}
 This definition is taken from \cite{gtt}. We then have the following.

\begin{lemma}
The restriction $r: \Sigma_{B_j} \to \mathbb{R}$ is lattice if and only if there exists $a,b \in \mathbb{R}$ such that for each vertex $v \in B_j$
$$\{ \tau(g)-a|g| : g \in L_v\} \subseteq b \mathbb{Z}.$$
\end{lemma}

\begin{proof}
Take $g \in L_v$. By the H\" older properties of $r$, we have that,
$$d(o,go) = r^{|g|}(x_g) + O(1)$$
where $x_g \in \Sigma_{B_j}$ is the periodic point obtained from repeating the loop corresponding to $g \in L_v$. The implied error is uniform in $g$. Applying this equality to $g^n$, using that $|g^n| = n|g|$ and then dividing by $n$ and letting $n$ tend to infinity shows that $\tau(g) = r^{|g|}(x_g)$. Substituting this expression into the non-lattice condition concludes the proof.
\end{proof}

We note that, if $r: \Sigma_{B_j} \to \mathbb{R}$ is non-lattice for any maximal component $B_j$, then by Lemma \ref{cohomcond} the variance $\sigma^2$ associated to the displacement function is strictly positive. We are now ready to prove our result.

\begin{proof} [proof of Theorem \ref{thmhad}]
We can use the same method used above to prove Theorem \ref{ndllt}. To apply our argument we need to show that the restrictions $r: \Sigma_{B_j} \to \mathbb{R}$ are non-lattice. Once we have shown this, our result follows as before. \\ \indent
To prove this non-lattice condition we use arguments due to Dal'bo that are used in \cite{db} to prove non arithmeticity of the length spectrum. We begin by noting that, by Corollary 6.11 of \cite{gtt}, for a vertex $v$ belonging to a maximal component, there exist independent hyperbolic elements $g,h \in L_v$ (i.e. $g$ and $h$ both have two fixed points in the boundary $\partial X$ and these four fixed points are all distinct). We now consider for $n\in \mathbb{Z}_{\ge 0}$ the elements $gh^n$.
These elements satisfy the following properties,
\begin{enumerate}
\item for all $n$ sufficiently large $gh^n$ is hyperbolic, and;
\item for all $n \in \mathbb{Z}_{\ge 0}$, $|gh^n|= |g| + n|h|$.
\end{enumerate}
The first property is easy to verify and the second follows from the properties of the coding from Definition $2.5$ . This second identity in the above implies that for any $a \in\mathbb{R}$,
\begin{equation} \label{iden}
 e^{\tau(gh^n) - \tau(gh^{n-1})} = e^{(\tau(gh^n)-a|gh^n|) - (\tau(gh^{n-1})- a|gh^{n-1}|) +a}
\end{equation}
for all $n\in\mathbb{Z}_{\ge 0}$. Furthermore, it is known (see \cite{db}) that
\begin{equation} \label{eq}
\lim_{n\to\infty} e^{\tau(gh^n) - \tau(gh^{n-1})} = e^{\tau(h)}.
\end{equation}
\indent We now suppose for contradiction that
$$\{\tau(g)-a|g|: g \in L_v\} \subseteq b\mathbb{Z}$$
for some $a,b \in\mathbb{R}$. By (\ref{iden}) we have that 
$$ e^{\tau(gh^n) - \tau(gh^{n-1})} \in e^{a + b\mathbb{Z}}$$
for all $n\in\mathbb{Z}_{\ge0}$. The convergence in (\ref{eq}) then implies that for all $n$ sufficiently large 
$$\tau(gh^n) = \tau(gh^{n-1}) + \tau(h).$$
However, from the proof of Proposition 2.1 in \cite{db}, we can find arbitrarily large $n$ such that that $\tau(gh^{n}) < \tau(gh^{n-1}) +\tau(h)$. This contradiction shows that the restrictions $r: \Sigma_{B_j} \to \mathbb{R}$ are non-lattice. By our above discussion, this concludes the proof.
\end{proof}

\section{Appendix}

In this section we prove Proposition $8.2$. The main ingredient is the aforementioned `Basic Inequality'.

\begin{proposition} [Basic Inequality \cite{wiley} Lemma $2$, Section $XVI.3$]
Suppose that $F$ is a probability distribution with vanishing expectation and Fourier transform $\widehat{F}$. Suppose that $N$ is the normal distribution with mean $0$, variance $\sigma^2>0$ and derivative $N'$. Suppose further that $F-N$ vanishes at $\pm \infty$. Then,
$$\| F - N\|_\infty \le \frac{1}{\pi} \int_{-T}^T \frac{1}{|t|} \left| \widehat{F}(t) - e^{-\sigma^2t^2/2} \right| dt + \frac{24\|N'\|_\infty}{\pi T},$$
where $T>0$ is arbitrary.
\end{proposition}
Take $H_n$, $E_n$, $T_n$, $C$ and $N$ as in the statement of Proposition $8.2$.
\begin{proof} [Proof of Proposition $8.2$]
Consider the distributions $F_n(x) := H_n(x+E_n)$. These have mean zero. Hence, by Proposition $12.1$
$$ \|F_n - N\|_\infty \le \frac{1}{\pi} \int_{-T_n}^{T_n}  \frac{1}{|t|} |e^{-itE_n}\widehat{H}_n(t)-e^{-\sigma^2t^2/2}| \ dt + \frac{24\|N'\|_\infty}{\pi T_n},$$
for all $n \in \mathbb{Z}_{\ge 0}$. We also have
\begin{align*}
\int_{-T_n}^{T_n} \frac{1}{|t|} |e^{-itE_n} \widehat{H}_n(t) - \widehat{H}_n(t)| \ dt &= \int_{-T_n}^{T_n} \frac{1}{|t|}|e^{-itE_n}-1| |\widehat{H}_n(t)| \ dt\\
&\le |E_n| e^{|T_n E_n|} \int_{-T_n}^{T_n} |\widehat{H}_n(t)| \ dt \\
&\le C |E_n| e^{|T_n E_n|},
\end{align*}
for all $n \in \mathbb{Z}_{\ge 0}$.
Now, define
$$M_n := \int_{-T_n}^{T_n} \frac{1}{|t|} |\widehat{H}_n(t)-e^{-\sigma^2t^2/2}| \ dt + |E_n|e^{|T_nE_n|} + \frac{1}{T_n}.$$
From the above,
$$\|F_n - N\|_\infty = O(M_n).$$
We then observe that
$$\|H_n - F_n\|_\infty \le \|N'\|_\infty |E_n| + 2M_n.$$
Lastly,
$$\|H_n - N\|_\infty \le \|F_n - N\|_\infty + \|H_n - F_n\|_\infty = O(M_n + |E_n|),$$
where the implied error term does not depend on $n \in \mathbb{Z}_{\ge 0}$. This is precisely the statement of Proposition $8.2$.
\end{proof}

\end{document}